\documentclass{article}
\usepackage{amssymb,amsmath,amsthm,graphicx}
\usepackage{amsfonts}
\usepackage{pinlabel}
\usepackage[dvipsnames,svgnames,table]{xcolor}
\usepackage{xcolor}
\newtheorem{theorem}{Theorem}
\newtheorem{lemma}[theorem]{Lemma}
\newtheorem{prop}[theorem]{Proposition}
\newtheorem{corollary}[theorem]{Corollary}

\theoremstyle{definition}

\newtheorem{question}{Question}

\newtheorem{problem}{Problem}
\date{}

\title{Bridge Positions of Links That Cannot Be Monotonically Simplified}

\author{Puttipong Pongtanapaisan\footnote{Email: ppongtan@asu.edu} \;\ and Daniel Rodman\footnote{Email: daniel$\_$rodman@taylor.edu}}
\begin{document}
\maketitle

\begin{abstract} 
For any pair of integers \( m \) and \(n \) such that \( 3 <m< n \), we provide an infinite family of links, where each link in the family has a locally minimal \( n \)-bridge position and a globally minimal \( m \)-bridge position. We accomplish this by applying the criterion of Takao et al. The \( n \)-bridge position is interesting because the corresponding bridge sphere is unperturbed, so it must be perturbed at least once before it can be de-perturbed to attain a globally minimal \(m\)-bridge sphere.
\end{abstract}

\section{Introduction}

\subsection{Preliminaries}\label{subsec:prelim}
A tangle \((B^3,T)\) is \textit{trivial} if there is a Morse function \(h:(B^3,T)\rightarrow (-\infty,0]\) such that $h^{-1}(0)=\partial B^3$, $h|_{B^3}$ has exactly one index zero critical point and $h$ restricted to each arc component of \(T\) has exactly one index zero critical point. A \textit{bridge position} of a link \(L\) in the 3-sphere \(S^3\) is a decomposition of \((S^3,L)\) into a union of two trivial tangles which intersect in a sphere called a \textit{bridge sphere}. In the literature, bridge positions are sometimes also referred to as \textit{bridge splittings} or \textit{bridge presentations}.
Any given link has infinitely many distinct bridge positions, and it is useful to assign a measure of complexity to each one.
The usual way to do this is to count the number of components in the tangle on one side of the bridge sphere.
This is called the \textit{bridge number}.

\begin{figure}
\centering
\includegraphics[width=.5\textwidth]{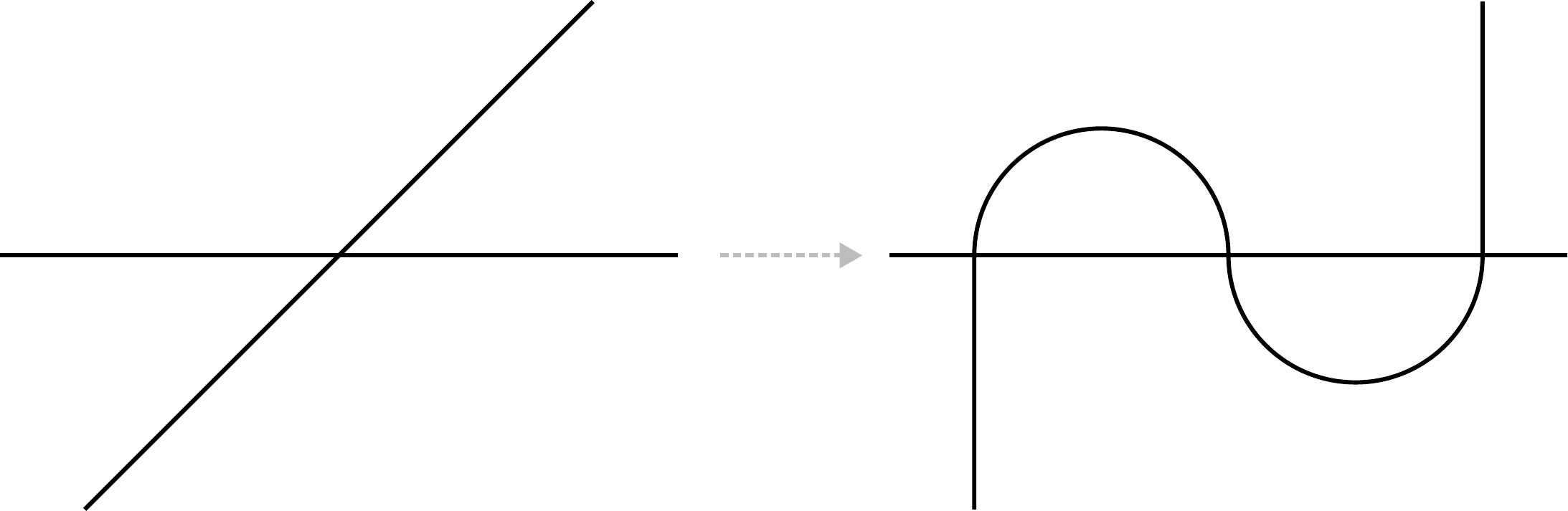}
\caption{Perturbation.}
\label{fig:241122_perturbation}
\end{figure}

Birman proved in \cite{birman1976stable} that any two bridge positions for a link are related by a sequence of moves called \textit{perturbations} (shown schematically in Figure \ref{fig:241122_perturbation}) and the inverse moves (de-perturbations). The horizontal line in Figure \ref{fig:241122_perturbation} represents a bridge sphere. Once a perturbation is performed, there exists a pair of disks on opposite sides of the bridge sphere intersecting each other in one point on the link, where the boundary of each disk can be decomposed into two parts: one part on the bridge sphere and one part on a subarc of the link.
Observe that a perturbation move increases the number of tangle components on either side of the bridge sphere by 1, resulting in a new bridge position with a higher bridge number than the original.

A sphere which does not admit a de-perturbation move is called \textit{unperturbed}, but it can be difficult to determine whether a given bridge sphere is unperturbed. 
In other words, if a link has an \(n\)-bridge position, it is not always clear whether it also has an \((n-1)\)-bridge position. 
The notion of strong irreducibility, which we will now define, is useful for this problem. 

For a surface $\Sigma$, the \textit{curve complex} $\mathcal{C}(\Sigma)$ is the graph whose vertices are isotopy classes of essential simple closed curves on $\Sigma$. Two vertices are connected by an edge in \(\mathcal{C}(\Sigma)\) if their representatives can be realized disjointly.
Given a link \(L\in S^3\) with bridge sphere \(\Sigma\), a disk in \(S^3\) is called \textit{essential} if 
1) its boundary is an element of \(\mathcal{C}(\Sigma)\) which cuts \(\Sigma\) into two punctured disks, and 
2) its interior is disjoint from both \(\Sigma\) and \(L\).
(Essential disks are also called \textit{compressing disks}.)
Let $\mathcal{D_+}$ (resp.\ $\mathcal{D}_-$) be the subcomplex of \(\mathcal{C}(\Sigma)\) spanned by curves on the bridge sphere that bound essential disks above (resp.\ below) the bridge sphere. 
The bridge sphere $\Sigma$ is \textit{strongly irreducible} if there is no edge connecting $\mathcal{D_+}$ to $\mathcal{D}_-$. A useful observation made in \cite{ozawa2013locally} is that if a bridge sphere is strongly irreducible, then the bridge sphere is unperturbed. We note that the converse does not necessarily hold, for there are unperturbed bridge spheres that are not strongly irreducible \cite{pongtanapaisan2020keen}.

\subsection{Moves Relating Representations}\label{subsec:moves}
In general there are numerous methods to represent a fixed link type. Some well-known examples include braid presentations and arc presentations as well as the bridge positions we discussed above \cite{cromwell2004knots}. In each of these methods for representing links, there are elementary moves that preserve both the presentation type and the topological type. 
As we saw above with bridge positions and perturbation moves, these elementary moves may alter the complexities of the presentations. 
An interesting problem is to measure how many moves one has to perform to relate two presentations of the same link type. The following question naturally arises.

\begin{question}\label{q:monotonic}
  Given a measure of complexity (e.g., braid index, arc index, bridge number, etc.), is there a sequence of complexity-decreasing moves from every diagram or embedding of a link to one of minimal complexity?
\end{question}

We refer to a presentation with minimal complexity as a \textit{global minimum}. A \textit{local minimum} is a presentation that does not allow for a complexity-decreasing move but is not necessarily a global minimum. For braid presentations and their corresponding elementary moves, there exists a presentation of the unknot that is a local minimum but not a global minimum \cite{morton1983irreducible}. On the other hand, there is no locally minimal bridge position for the unknot other than the unique global minimum \cite{otal1982presentations}. Moreover, regarding bridge positions, the answer to Question~\ref{q:monotonic} is affirmative for 2-bridge knots \cite{otal1985presentations}, torus knots \cite{ozawa2011nonminimal}, and the $(p, q)$-cable of an $mp$-small knot $K$, where every non-minimal bridge position of $K$ is perturbed \cite{zupan2011properties}.

There is an interesting topological consequence if the answer to Question~\ref{q:monotonic} is \textit{no} for a given bridge position: There will be at least two nonisotopic bridge spheres with the same Euler characteristic for the link. A result of Maggy Tomova then implies that the Euler characteristic of one bridge sphere can be used to estimate the \textit{bridge distance} of the other \cite{tomova2007multiple}. The bridge distance is a measure of complexity that can be difficult to compute, but it provides information about hyperbolicity of link complements \cite{bachman2005distance}, the minimum number of special generators for knot groups \cite{blair2024adding}, and distortion of knots \cite{blair2020distortion}.

Numerous results provide families of bridge positions for which the answer to Question~\ref{q:monotonic} is \textit{no}. For instance, Birman and Montesinos use algebraic techniques and branched coverings to provide 3-bridge links with two distinct global minima \cite{birman1976stable,montesinos1976minimal}. Jang later gave examples of 3-bridge links with four distinct global minima \cite{jang2013classification}. The first example of a locally minimal, but not globally minimal bridge position was constructed in \cite{ozawa2013locally}. In that paper, the author gave one hyperbolic 3-bridge knot as the example and posed the following question:

\begin{problem}[Problem 3.2 of \cite{ozawa2013locally}]\label{problem:main}
    For an integer $n > 3$, can we generate infinitely many $n$-bridge
positions which are locally minimal, but not globally minimal?
\end{problem}

For any two distinct arbitrarily large integers \(n\) and \(m\), we answer Problem \ref{problem:main} affirmatively by constructing an infinite family of links \(L_{m,n}\) that have both a locally minimal \(n\)-bridge position and a locally minimal \(m\)-bridge position.

\begin{theorem}\label{thm:main}
For any two integers \(m\) and \(n\) such that \(3<m<n\), there exists an infinite family of links, each of which has both an unperturbed \(n\)-bridge sphere and an unperturbed \(m\)-bridge sphere.
\end{theorem}

Our proof will be constructive.
Each link \(L_{m,n}\) in our construction will have exactly \((m-1)\) components, and we remark that our examples are new since they have multiple components. Also, $m=3$ for the examples provided in \cite{jang2016knot}, but $m$ can be arbitrary in this paper. 
There are some quick corollaries we can draw. The main result of \cite{lee2023nonminimal} and \cite{zupan2011properties} implies the following.

\begin{corollary} 
The link $L_{m,n}$ is a prime link that is not an iterated torus knot or a $(2,2q)$-cable of an iterated torus knot.
\end{corollary}

We also have a lower bound on the number of moves it takes to go between the two locally minimal bridge positions given by Theorem \ref{thm:main}.

\begin{corollary}
    It takes at least $n-m+2$ perturbation/de-perturbation moves to obtain one bridge sphere in Theorem \ref{thm:main} from the other one.
\end{corollary}

\section{The links of interest}\label{sec:links}

\begin{figure}
\labellist
\small\hair 2pt

\pinlabel {-1} [Br] at 10 30
\pinlabel {0} [Br] at 10 58.375
\pinlabel {1} [Br] at 10 86.75
\pinlabel {2} [Br] at 10 115.125
\pinlabel {3} [Br] at 10 143.5
\pinlabel {4} [Br] at 10 171.875
\pinlabel {5} [Br] at 10 200.25
\pinlabel {6} [Br] at 10 228.625
\pinlabel {7} [Br] at 10 257
\pinlabel {\(2m-7\)} [Br] at 10 340.4
\pinlabel {\(2m-6\)} [Br] at 10 368.7
\pinlabel {\(2m-5\)} [Br] at 10 397
\pinlabel {\(2m-4\)} [Br] at 10 425.3
\pinlabel {\(2m-3\)} [Br] at 10 453.6
\pinlabel {\(2m-2\)} [Br] at 10 481.9
\pinlabel {\(2m-1\)} [Br] at 10 510.1
\pinlabel {\(y\)} [Br] at 10 554.1

\pinlabel {1} [t] at 47.5 10
\pinlabel {2} [t] at 75.86 10
\pinlabel {3} [t] at 104.21 10
\pinlabel {4} [t] at 132.57 10
\pinlabel {5} [t] at 160.93 10
\pinlabel {6} [t] at 189.29 10
\pinlabel {7} [t] at 217.64 10
\pinlabel {8} [t] at 246 10
\pinlabel {\(2n\)} [t] at 615 10
\pinlabel {\(x\)} at 665 35

\endlabellist
\centering
\includegraphics[width=.6\textwidth]{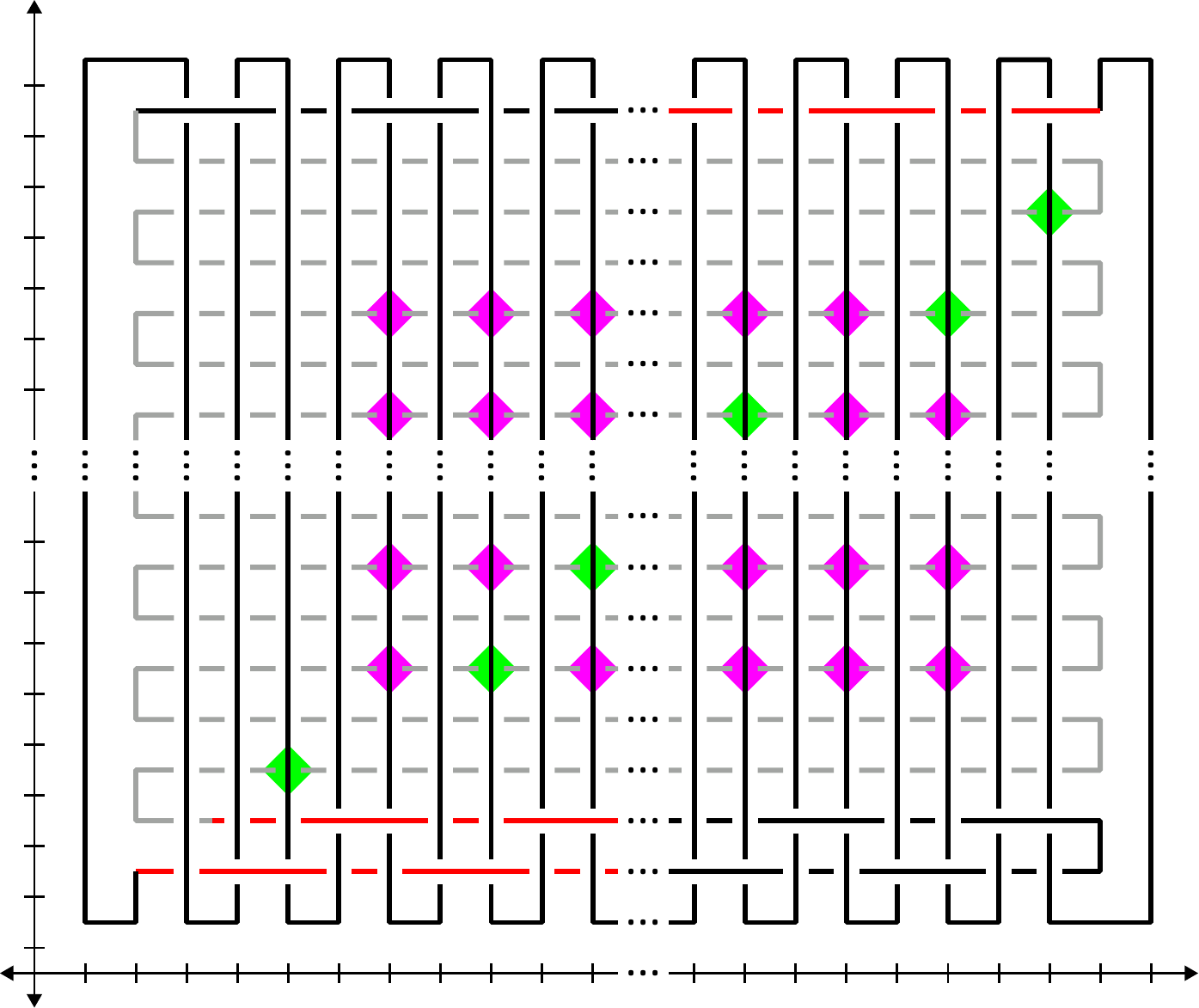}
\caption{The knot diagram \(D_{m,n}\) when \(n\) is even.
When \(n\) is odd, the crossing information is switched at every crossing involving a red colored strand.
An example collection of distinguished crossings is highlighted in green.}
\label{fig:241115_starting-diagram}
\end{figure}

\begin{figure}
\centering
\includegraphics[width=.3\textwidth]{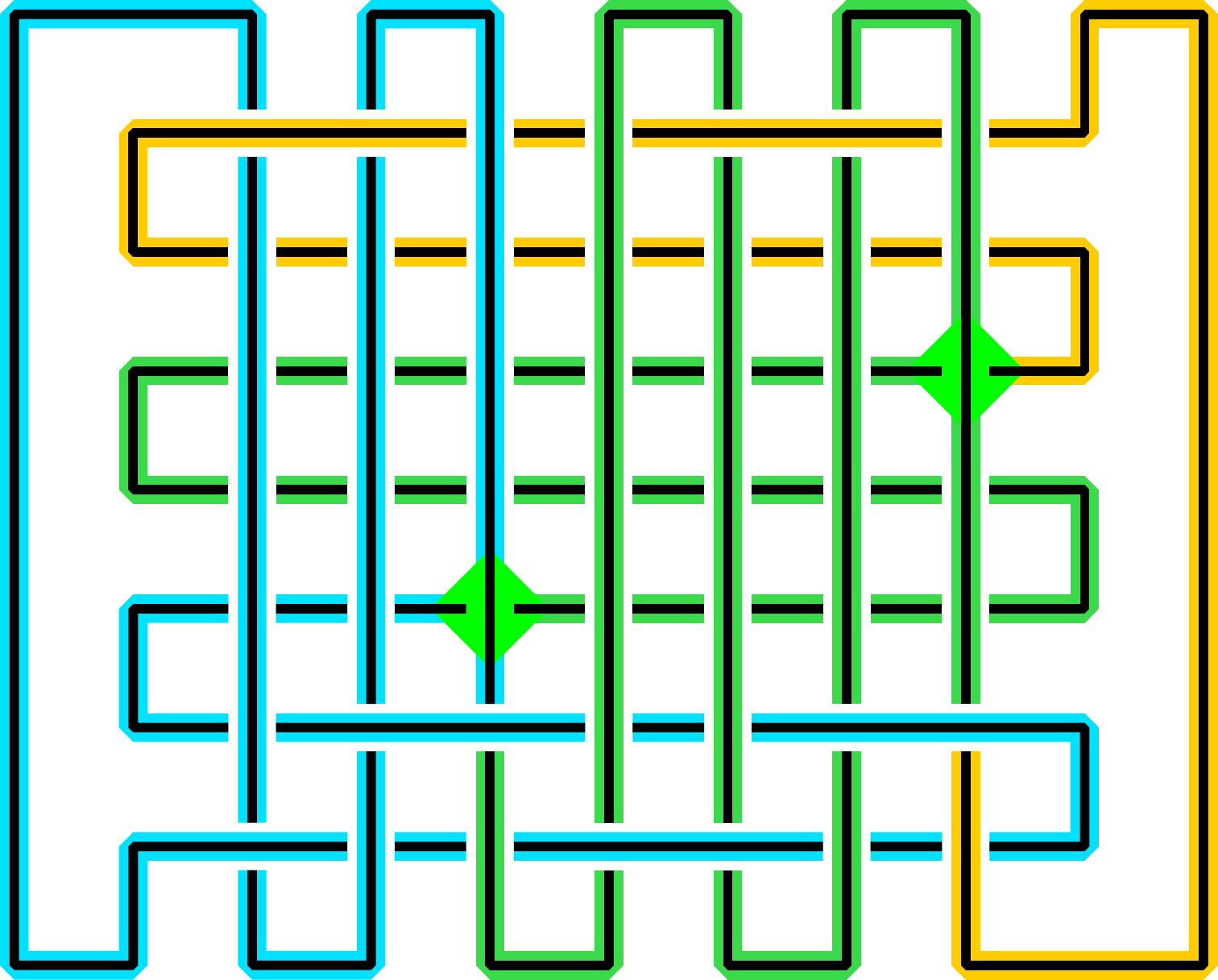}
\caption{The smallest diagram \(D_{4,5}\).
The two distinguished crossings are highlighted in green, and subarcs of the diagrams are highlighted different colors, illustrating the link components they will belong to after surgery.}
\label{fig:241126_low-value-links}
\end{figure}

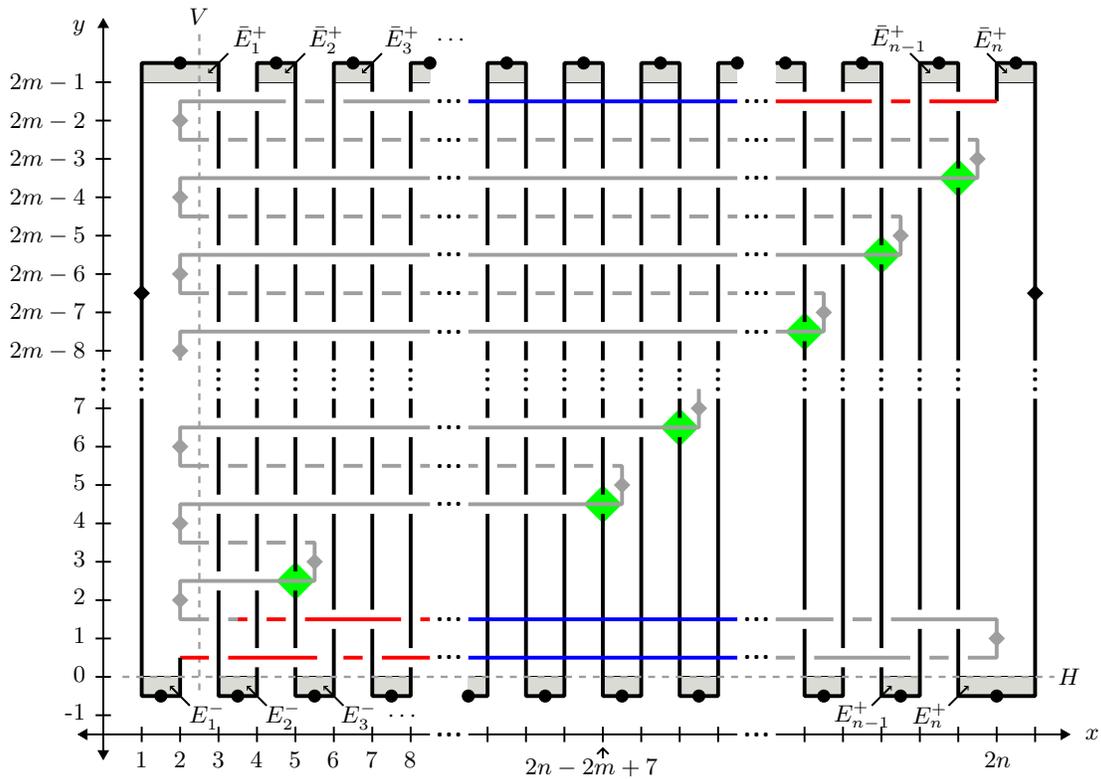
\begin{figure}
\labellist
\small\hair 2pt

\pinlabel {-1} [Br] at 10 30
\pinlabel {0} [Br] at 10 58.375
\pinlabel {1} [Br] at 10 86.75
\pinlabel {2} [Br] at 10 115.125
\pinlabel {3} [Br] at 10 143.5
\pinlabel {4} [Br] at 10 171.875
\pinlabel {5} [Br] at 10 200.25
\pinlabel {6} [Br] at 10 228.625
\pinlabel {7} [Br] at 10 257
\pinlabel {\(2m-8\)} [Br] at 10 298
\pinlabel {\(2m-7\)} [Br] at 10 326.29
\pinlabel {\(2m-6\)} [Br] at 10 354.57
\pinlabel {\(2m-5\)} [Br] at 10 382.86
\pinlabel {\(2m-4\)} [Br] at 10 411.14
\pinlabel {\(2m-3\)} [Br] at 10 439.43
\pinlabel {\(2m-2\)} [Br] at 10 467.71
\pinlabel {\(2m-1\)} [Br] at 10 496
\pinlabel {\(y\)} [Br] at 10 540

\pinlabel {1} [t] at 47.5 10
\pinlabel {2} [t] at 75.86 10
\pinlabel {3} [t] at 104.21 10
\pinlabel {4} [t] at 132.57 10
\pinlabel {5} [t] at 160.93 10
\pinlabel {6} [t] at 189.29 10
\pinlabel {7} [t] at 217.64 10
\pinlabel {8} [t] at 246 10
\pinlabel {\(2n-2m+7\)} [t] at 380 5
\pinlabel {\(2n\)} [t] at 680 10
\pinlabel {\(x\)} [t] at 750 27

\pinlabel {\(\bar{E}_1^+\)} at 128 533
\pinlabel {\(\bar{E}_2^+\)} at 184 533
\pinlabel {\(\bar{E}_3^+\)} at 240 533
\pinlabel {\(\cdots\)} at 276 533
\pinlabel {\(\bar{E}_{n-1}^+\)} at 607 533
\pinlabel {\(\bar{E}_{n}^+\)} at 675 533

\pinlabel {\(E_1^-\)} at 96 33
\pinlabel {\(E_2^-\)} at 152 33
\pinlabel {\(E_3^-\)} at 208 33
\pinlabel {\(\cdots\)} at 240 33
\pinlabel {\(E_{n-1}^+\)} at 580 33
\pinlabel {\(E_{n}^+\)} at 630 33

\pinlabel {\(H\)} at 733 62
\pinlabel {\(V\)} at 90 550

\endlabellist
\centering
\includegraphics[width=.8\textwidth]{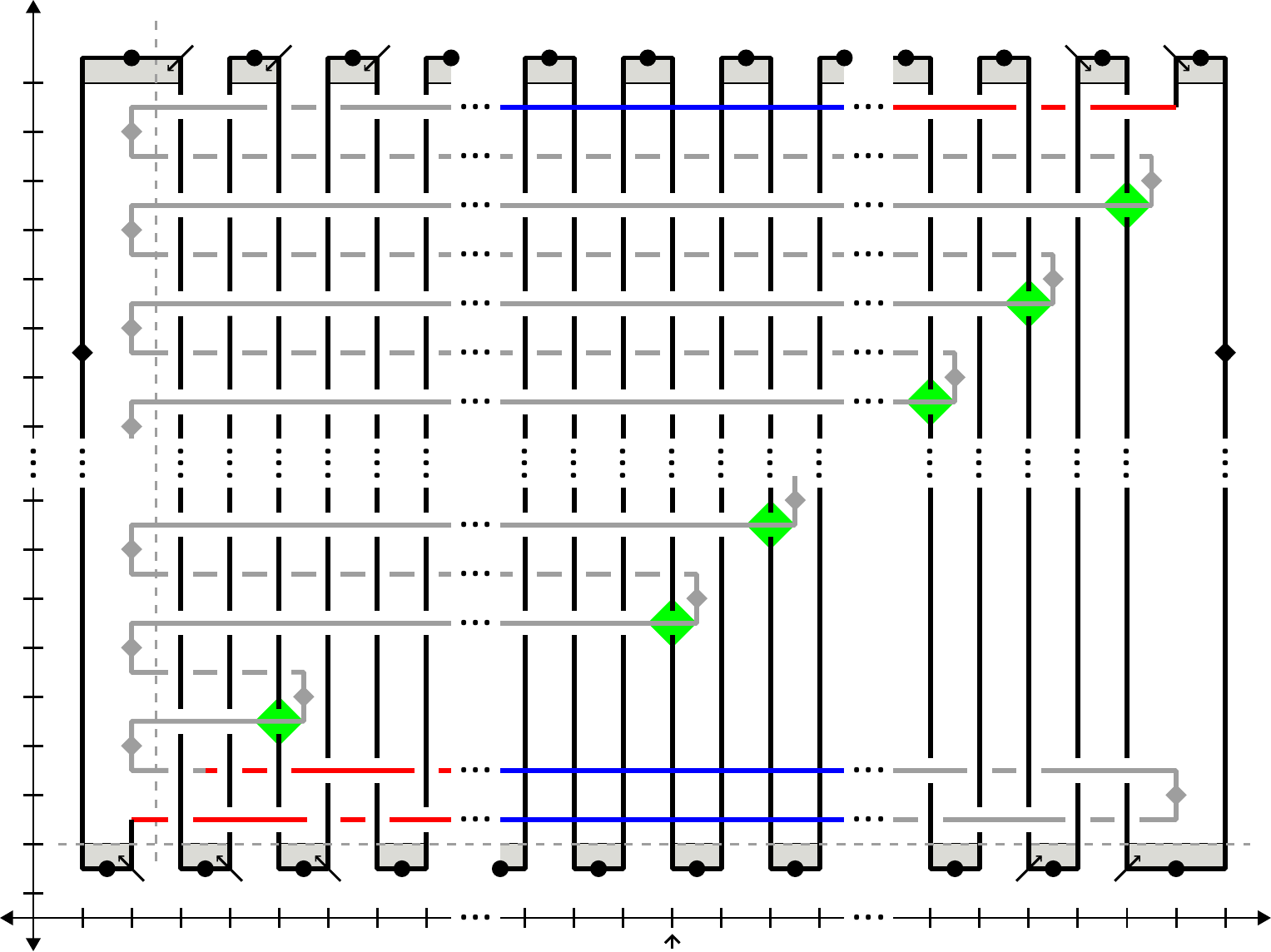}
\caption{The knot \(K_{m,n}\) with \(n\) bridges when \(n\) is even.
In the case when \(n\) is odd, the diagram looks the same except the crossing information is switched at every crossing involving a red arc.
Along the blue arcs, the crossing information is not shown, but it follows the over-over-under-under pattern of the strands to the left and right.
The distinguished crossings are highlighted in green.}
\label{fig:241007_Knot}
\end{figure}

The delightful ideas and knot construction in \cite{jang2016knot} (especially sections 4 and 5) are the primary inspiration and jumping-off point for us in this paper, and we assume readers will read our work after or at least with reference to that one.
To aid readers therefore, much of our notation will be the same as or similar to the notation in \cite{jang2016knot}.

We begin by constructing a knot diagram \(D_{m,n}\).
Then we will switch several crossings of \(D_{m,n}\) and embed it into \(S^3\), naming the resulting knot \(K_{m,n}\) (which is inspired by the knot \(K_n\) from \cite{jang2016knot}).
We will then modify \(K_{m,n}\) by replacing certain crossings with twist regions, thereby obtaining a link which we call \(L_{m,n}\).
It is this link which we will show satisfies Theorem \ref{thm:main}.
(Technically \(L_{m,n}\) represents an infinite family of links.)
A fundamental insight from \cite{jang2016knot} is that by choosing particular crossing information along the top two horizontal strands of their knot \(K_n\) and along its two bottom horizontal strands, we guarantee that the knot satisfies the 2-connected condition (which we will define in Section \ref{subsec:two-conned}).
Thus we will employ the same crossing patterns in the construction of \(L_{m,n}\).

Fix positive integers \(m\) and \(n\) such that \(3<m<n\).
For \(n\) even, \(D_{m,n}\) is the knot diagram in Figure \ref{fig:241115_starting-diagram}.
When \(n\) is odd, \(D_{m,n}\) is the same except that the crossing information is switched at every crossing involving a red colored strand.
Observe that the gray colored strands always cross under other strands.
The diagram \(D_{n,m}\) is pictured in the \(xy\)-plane, which allows us to easily talk about the several crossings which are highlighted (green and pink) in the figure.
Those crossings occur at the points \({(5,2.5)}\) and \({(2n-2, 2m-3.5)}\) and at every point of the form \({(i,j)}\), where \(i\in\{7,9,11,\hdots,2n-3\}\) and \(j\in\{4.5, 6.5, 8.5, 2m-5.5\}\).

Observe that there are \(m-2\) rows of highlighted crossings.
(The top and bottom rows each only contain a single highlighted crossing.)
Create a collection of \textit{distinguished crossings} by selecting one crossing from each row in such a way that for any pair of distinguished crossings, one of them is both higher and to the right of the other.
(In our figures, distinguished crossings are highlighted in green.
For ease of illustration, Figures \ref{fig:241007_Knot}, \ref{fig:241028_9-6-example}, and \ref{fig:241016_9-6-example-rotated.pdf} show a particular selection of distinguished crossings located as far to the right as possible.)
The smallest values we allow for \(m\) and \(n\) are \(m=4\) and \(n=5\), and an image of \(D_{4,5}\) is provided in Figure \ref{fig:241126_low-value-links}, along with its two distinguished crossings.

Next we modify \(D_{m,n}\) by switching the crossing information at each distinguished crossing and at every crossing along the gray strand to the left of a distinguished crossing.
Then we perform a series of Reidemeister II moves to remove unnecessary crossings involving the gray and black strands, and we embed this modified diagram in \(S^3=R^3\cup\{\infty\}\) as shown in Figure \ref{fig:241007_Knot} and call the resulting knot \(K_{m,n}\).
We stipulate that 1) the black colored arcs of \(K_{m,n}\) are located in the \(xy\)-plane, 2) \(K_{m,n}\subset{[0,2n+2]}\times{[-1,2m]}\times{[-1,1]}\), and 3) the gray colored arc is in minimal position with respect to the \(xy\)-plane.

The \(y\)-axis induces a natural height function \(h_y\) on \(K_{m,n}\) with respect to which the horizontal dashed line represents the level sphere \(h_y^{-1}(0)\), which we name \(H\).
Similarly, the \(x\)-axis induces a height function \(h_x\) on the knot, and the vertical dashed line represents a level sphere \(h_x^{-1}(2.5)\), which we call \(V\).
After a slight perturbation of the knot, both \(H\) and \(V\) are simultaneously bridge spheres with respect to \(h_y\) and \(h_x\), respectively, with the circular points in the figure representing the knot's extrema with respect to \(H\) and with the diamond points representing the extrema with respect to \(V\).
To distinguish between all these extrema, we will call the circular points \textit{\(H\)-maxima} and \textit{\(H\)-minima}, and we will call the diamond points \textit{\(V\)-maxima} and \textit{\(V\)-minima}. 
We have constructed \(K_{m,n}\) so that \(n\) (resp., \(m\)) is the number of \(H\)-maxima (resp., \(V\)-maxima).
Intuitively speaking, the knot \(K_{m,n}\) is similar to \(K_n\) from \cite{jang2016knot}, except that \(K_{m,n}\) includes a number of extra horizontal loops of various lengths which are not present in \(K_n\).

\begin{figure}
\labellist
\small\hair 2pt
\pinlabel{Surgery} at 220 113
\pinlabel{Isotopy} at 520 113
\endlabellist
\centering
\includegraphics[width=.4\textwidth]{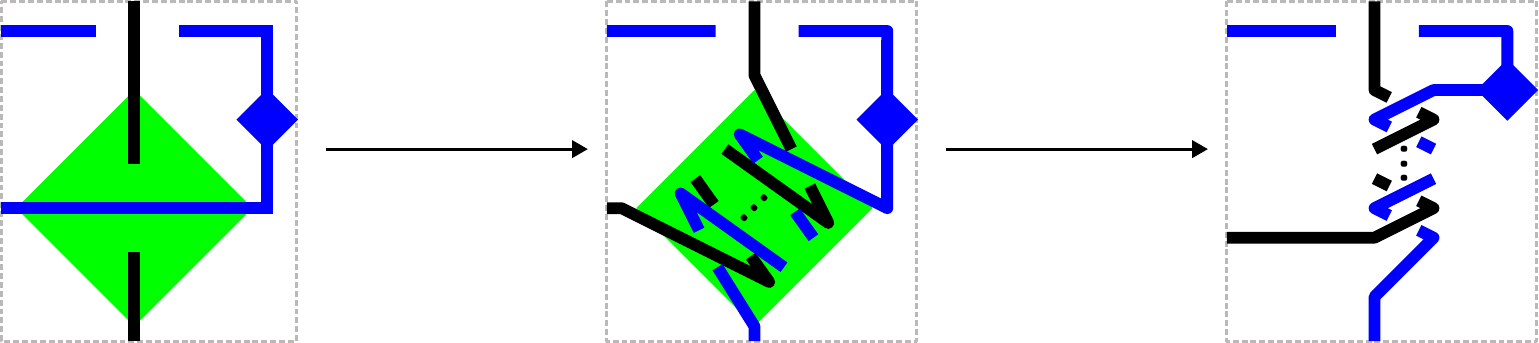}
\caption{At each distinguished crossing we perform a surgery, replacing the crossing with a twist region as pictured here.
This changes the knot \(K_{m,n}\) into the link \(L_{m,n}\).
We perform the isotopy on the right, making \(H\) an \(n\)-bridge sphere for \(L_{m,n}\).
On the right side of Figure \ref{fig:241028_9-6-example}, we more explicitly show how the twist region and the \(V\)-extrema are embedded with respect to the height function \(h_y\).}
\label{fig:241028_rotating-twists}
\end{figure}

\begin{figure}
\labellist
\small\hair 2pt
\pinlabel {-1} [Br] at 50 97
\pinlabel {0} [Br] at 50 182.5
\pinlabel {1} [Br] at 50 268
\pinlabel {2} [Br] at 50 353.5
\pinlabel {3} [Br] at 50 439
\pinlabel {4} [Br] at 50 524.5
\pinlabel {5} [Br] at 50 610
\pinlabel {6} [Br] at 50 695.5
\pinlabel {7} [Br] at 50 781
\pinlabel {8} [Br] at 50 866.5
\pinlabel {9} [Br] at 50 952
\pinlabel {10} [Br] at 50 1037.5
\pinlabel {11} [Br] at 50 1123
\pinlabel {\(y\)} [Br] at 50 1208.5

\pinlabel {\(y\)} [l] at 2725 1072
\pinlabel {\(s_0+1\)} [l] at 2725 935
\pinlabel {\(s_0+0.5\)} [l] at 2725 793.8
\pinlabel {\(s_0\)} [l] at 2725 652.5
\pinlabel {\(s_0-0.5\)} [l] at 2725 511.3
\pinlabel {\(s_0-1\)} [l] at 2725 370

\pinlabel {1} [t] at 152 40
\pinlabel {2} [t] at 237.3 40
\pinlabel {3} [t] at 322.7 40
\pinlabel {4} [t] at 408 40
\pinlabel {5} [t] at 493.3 40
\pinlabel {6} [t] at 578.7 40
\pinlabel {7} [t] at 664 40
\pinlabel {8} [t] at 749.3 40
\pinlabel {9} [t] at 834.7 40
\pinlabel {10} [t] at 920 40
\pinlabel {11} [t] at 1005.3 40
\pinlabel {12} [t] at 1090.7 40
\pinlabel {13} [t] at 1176 40
\pinlabel {14} [t] at 1261.3 40
\pinlabel {15} [t] at 1346.7 40
\pinlabel {16} [t] at 1432 40
\pinlabel {17} [t] at 1517.3 40
\pinlabel {18} [t] at 1602.7 40
\pinlabel {19} [t] at 1688 40
\pinlabel {\(x\)} [t] at 1750 40

\endlabellist
\includegraphics[width=.92\textwidth]{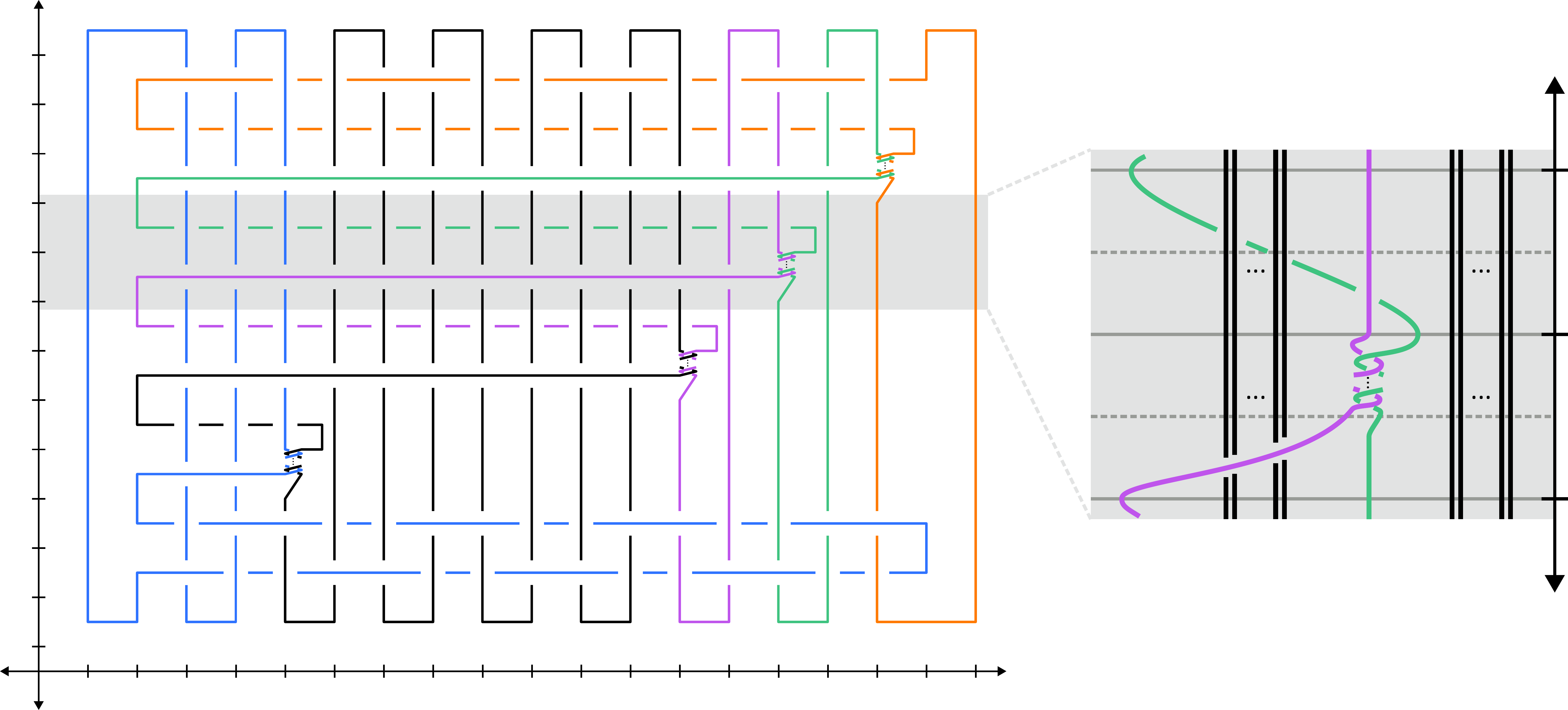}
\caption{On the left, an example of \(L_{6,9}\), with each component a different color.
On the right, a more detailed view of a horizontal strip of the link diagram, representative of how we situate each of the twist regions with respect to the \(y\)-axis.
We embed the link so that if before the surgery, a distinguished crossing was at height \(s_0-0.5\) for some integer height \(s_0\), then after the surgery, the twist region is contained vertically between the heights of \(s_0-0.5\) and \(s_0\).
The vertical black lines in that picture just represent some unspecified number of vertical strands of the link.}
\label{fig:241028_9-6-example}
\end{figure}
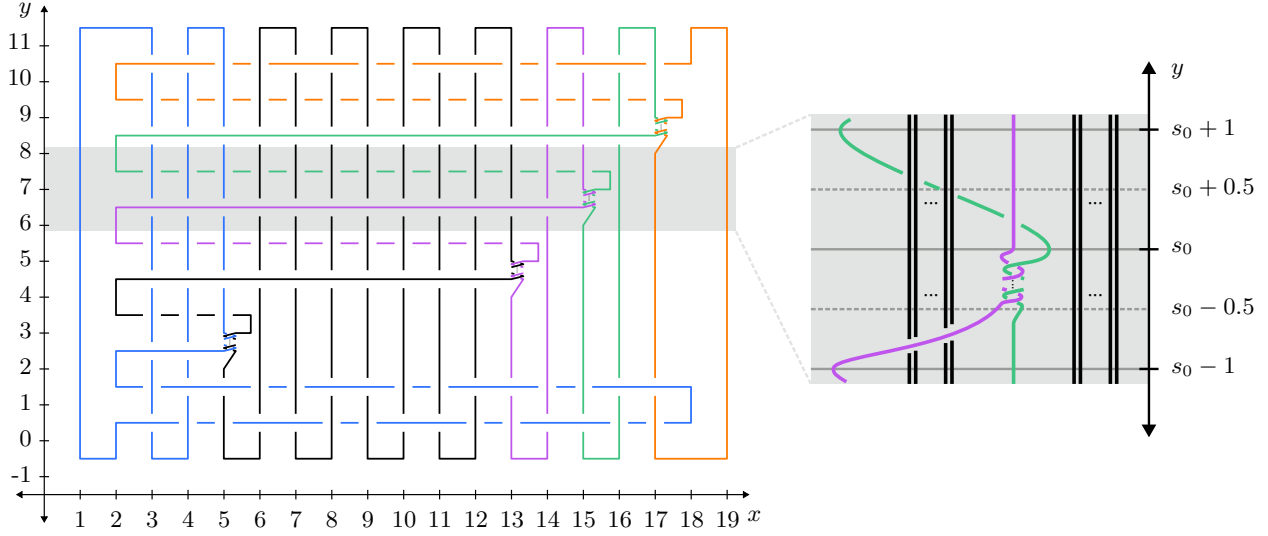

Next we replace each distinguished crossing with a twist region as shown in Figure \ref{fig:241028_rotating-twists}.
To be more precise, we will perform \(m-2\) surgeries on \(K_{m,n}\), changing it into a link, as follows.
We identify \(m-2\) small disjoint balls, each of which contains a 2-component untangle of \(K_{m,n}\) such that the projection in the \(z\)-direction of that untangle to the \(xy\)-plane contains exactly one distinguished crossing and no other crossings.
We replace each of these balls with another ball containing a right handed twist region with an even number of half twists, as depicted in Figure \ref{fig:241028_rotating-twists}.
The result of these surgeries is a link with \(m-1\) components, which we call \(L_{m,n}\).
An example of \(L_{6,9}\) is shown in Figure \ref{fig:241028_9-6-example}.

In Section \ref{sec:WRT-H}, we use the 2-connected condition to show that \(H\) is an unperturbed \(n\)-bridge sphere for \(L_{m,n}\), and in Section \ref{sec:WRT-V}, we show that \(V\) is an unperturbed \(m\)-bridge sphere for \(L_{m,n}\).

\section{\(L_{m,n}\) is Vertically Unperturbed}\label{sec:WRT-H}

\subsection{More Definitions}

The sphere \(H\) cuts \(S^3\) into two balls, which we call \(B^+\) and \(B^-\) (with the \(H\)-maxima contained in \(B^+\)).
As a bridge sphere, \(H\) cuts \(L_{m,n}\) into two trivial tangles.
The tangle contained in \(B^+\) we call \(\tau^+\), and likewise, \(\tau^-\) is the tangle contained in \(B^-\).

We will be interested in the sphere \(H\) as we perform isotopy, pushing it up and down between the \(H\)-maxima and \(H\)-minima.
To make this more formal, for each \(s\in{[0,2m-1]}\) let \({H}(s)\) be a sphere containing the horizontal square disk \({[0,2n+2]}\times\{s\}\times{[-1,1]}\).
We choose this collection of spheres in such a way that they give us a smooth isotopy from \(H=H(0)\) to \(H(2m-1)\).

At certain heights, all of the link's punctures are contained in a straight line.
These heights are given by the set \(S=\{0,1,2,\hdots,2m-1\}\cup\{2.5, 3.5, 4.5,\hdots,2m-3.5\}\).
(These are the integer heights and the heights directly below each twist region.)
For each height \(s\in S\), let \(\ell(s)\) be a simple loop in \(H(s)\) with the property that \(\ell(s)\) intersects the horizontal square disk \({[0,2n+2]}\times\{s\}\times{[-1,1]}\) in a straight line segment from the point \({(0,s,0)}\) to the point \({(2n+2,s,0)}\).
Thus for each \(s\in S\), all the punctures of \(L_{m,n}\cap H(s)\) are contained in \(\ell(s)\).
From left to right, we will name these punctures \(p_1(s)\), \(p_2(s)\), \(p_3(s)\), \(\hdots\), \(p_{2n}(s)\).
That is, \(p_i(s)\) is the \(i\)\textsuperscript{th} puncture along \(\ell(s)\).
Let \(\ell_i(s)\) be the closed arc of \(\ell(s)\) whose endpoints are \(p_i(s)\) and \(p_{i+1}(s)\) (where the addition is taken mod \(2n\)) and whose interior contains no other punctures.
The loop \(\ell(s)\) cuts \(H(s)\) into two hemispheres.
Let \(H_+(s)\) be the hemisphere containing the rectangular disk \({[0,2n+2]}\times\{s\}\times{[0,1]}\), and let \(H_-(s)\) be the other hemisphere.
Set \(H_+=H_+(0)\) and \(H_-=H_-(0)\).

For each \(i\in\{1,2,3,\hdots,n\}\), let \(\bar{E}_i^+\) and \(E_i^-\) be the gray disks shown in Figure \ref{fig:241007_Knot}, let \(\sigma_i^-=E_i^-\cap H\), and let \(\sigma_i^+(2m-1)=\bar{E}_i^+\cap H(2m-1)\).
For each integer \(s\in\{0,1,2,\hdots,2m-1\}\), the isotopy taking \(H(2m-1)\) to \(H(s)\) restricts to an isotopy from \(\sigma_i^+(2m-1)\) to an arc in \(H(s)\).
We call this image-arc \(\sigma_i^+(s)\), and we set \(\sigma_i^+=\sigma_i^+(0)\).
As we follow \(\sigma_i^+(2m-1)\) down through this isotopy, it traces out a disk with four sides (a disk which contains the arc \(\sigma_i^+(s)\) for all \(s\)). 
The top side of this traced-out disk is \(\sigma_i^+(2m-1)\), the bottom side is \(\sigma_i^+\), and the other two sides are subarcs of \(L_{m,n}\).
We will set \(E_i^+\) equal to the union of this disk with \(\bar{E}_i^+\).
Observe that the disks in \(\{E_i^+\}_i\) are pairwise disjoint, and the sets \(\{E_i^+\}_i\) and \(\{E_i^-\}_i\) are both complete collections of bridge disks for \(H\).

Let \(\sigma^+=\bigcup_{i=1}^n\sigma_i^+\) and \(\sigma^-=\bigcup_{i=1}^n\sigma_i^-\).
The loop \(\ell\) contains \(\sigma^-\) by construction, and we  may assume that \(\sigma^+\) is in minimal position with respect to \(\ell\).
Observe that \(\sigma^+\) and \(\sigma^-\) together comprise a bridge diagram for the pair \({(H,L_{m,n})}\).

\subsection{The 2-connected Condition and \(n\)-bundles}\label{subsec:two-conned}

A graph is \textit{2-connected} if for each vertex \(v\), the graph remains connected after removing \(v\) and all edges incident to \(v\).

In \cite{jang2016knot}, Jang et al.\ present a criterion on the pair \({({(\sigma^+,\sigma^-)},\ell)}\) which, when satisfied, guarantees that the corresponding bridge sphere is strongly irreducible.
They call this criterion the \textit{2-connected condition}, which we define below.

Observe that \(\sigma^-\) lies along the shared boundary of \(H_+\) and \(H_-\).
Fix \(i,j\in\{1,2,\hdots,n\}\) with \(i\neq j\), and fix \(\varepsilon\in\{+,-\}\).
Observe that \(\sigma^+\cap H_{\varepsilon}\) consists of a collection of arcs properly embedded in \(H_{\varepsilon}\), each of which is a subarc of \(\sigma_u^+\) for some \(u\). 
Each component of \(\sigma^+\cap H_{\varepsilon}\) cuts \(H_{\varepsilon}\) into two subdisks.
Let \(\mathcal{A}_{i,j,\varepsilon}\) be the subcollection of these arcs which separate \(\sigma_i^-\) in one subdisk from \(\sigma_j^-\) in the other subdisk.

Next we construct a simple\footnote{A simple graph contains no multiedges.} graph \(\mathcal{G}_{i,j,\varepsilon}\) from \(\mathcal{A}_{i,j,\varepsilon}\) as follows: Let the vertices of \(\mathcal{G}_{i,j,\varepsilon}\) be the integers \(1, 2,\hdots,n\), and let distinct vertices \(v\) and \(w\) be connected by an edge in \(\mathcal{G}_{i,j,\varepsilon}\) if and only if two components of \(\mathcal{A}_{i,j,\varepsilon}\) are a subarc \(\bar{\sigma}_v^+\subset\sigma_v^+\) and a subarc \(\bar{\sigma}_w^+\subset\sigma_w^+\) which are adjacent to each other.
By adjacent, we mean that \(H_{\varepsilon}\backslash\mathcal{A}_{i,j,\varepsilon}\) has a component whose closure contains both \(\bar{\sigma}_v^+\) and \(\bar{\sigma}_w^+\).
Now we are prepared to state the 2-connected condition from \cite{jang2016knot}: the bridge diagram \(\left(\left(\sigma^+,\sigma^-\right), \ell\right)\) satisfies the 2-connected condition if the graph \(\mathcal{G}_{i,j,\varepsilon}\) is 2-connected for all \(i,j,\varepsilon\).

The following theorem is due to Jang et al.\ in \cite{jang2016knot}.

\begin{theorem}[Jang, Kobayashi, Ozawa, and Takao]\label{thm:two-con-str-irr}
An \(n\)-bridge sphere \(S\) with \(n\geq 3\) of a link is strongly irreducible if there exist a bridge diagram \(\left(\sigma^+,\sigma^-\right)\) of the bridge sphere \(S\) and a loop \(\ell\) on \(S\) such that the pair \(\left(\left(\sigma^+,\sigma^-\right), \ell\right)\) satisfies the \(2\)-connected condition.
\end{theorem}

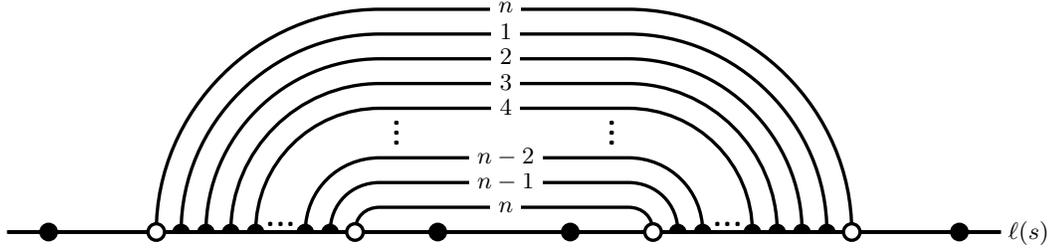
\begin{figure}
\labellist
\small\hair 2pt
\pinlabel {\(\ell(s)\)} [l] at 853 8
\pinlabel {\colorbox{white}{\(n\)}} at 427 30
\pinlabel {\colorbox{white}{\(n-1\)}} at 427 51.5
\pinlabel {\colorbox{white}{\(n-2\)}} at 427 73
\pinlabel {\colorbox{white}{4}} at 427 116
\pinlabel {\colorbox{white}{3}} at 427 137.5
\pinlabel {\colorbox{white}{2}} at 427 159
\pinlabel {\colorbox{white}{1}} at 427 180.5
\pinlabel {\colorbox{white}{\(n\)}} at 427 202
\endlabellist
\centering
\includegraphics[width=.8\textwidth]{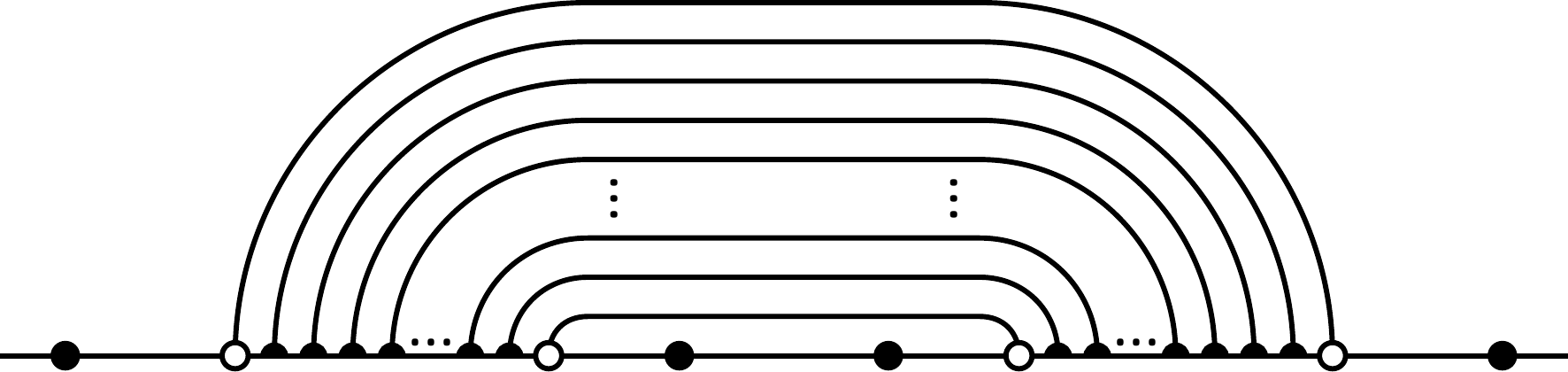}
\caption{An example of an \(n\)-bundle. 
An arc marked with a number \(i\) is a subarc of \(\sigma_i^+(s)\).
Punctures (i.e., points of \(\ell(s)\cap L_{m,n}\)) are marked with black circles.
The arc endpoints marked with half circles are not punctures.
Of the arc endpoints marked with white circles, at most one of them is a puncture.}
\label{fig:241112_n-bundle}
\end{figure}

It is known that a strongly irreducible bridge sphere is unperturbed, so to ultimately show that \(H\) is unperturbed, we need to show that \(\left(\left(\sigma^+,\sigma^-\right), \ell\right)\) satisfies the \(2\)-connected condition, which by Theorem \ref{thm:two-con-str-irr} implies \(H\) is strongly irreducible and therefore unperturbed.
Before we do that, we will construct one more definition which will be useful for Lemma \ref{lem:inductive}.
Fix a height \(s\) and a sign \(\varepsilon\in\{+,-\}\).
An \textit{\(n\)-bundle in \(H_{\varepsilon}(s)\)} (or just an \textit{\(n\)-bundle} when the context is clear) is a collection of arc components of \(\sigma^+(s)\cap H_{\varepsilon}(s)\) such that
\begin{enumerate}
\item for each \(i\in\{1,2,\hdots,n\}\), the collection includes a subarc of \(\sigma_i^+(s)\) and a subarc of \(\sigma_{i+1}^+(s)\) which are adjacent to each other in \(H_{\varepsilon}(s)\) (where the addition is taken mod \(n\)), 
\item there exist distinct \(k_1,k_2\in\{1,2,\hdots,2n\}\) such that each arc in the collection has one endpoint on \(\ell_{k_1}(s)\) and the other endpoint on \(\ell_{k_2}(s)\), and 
\item of all the endpoints of the arcs in this collection, at most one of them is a puncture.
\end{enumerate}
An example of an \(n\)-bundle is illustrated in Figure \ref{fig:241112_n-bundle}.
We can represent an \(n\)-bundle with a single arc whose endpoints are on \(\ell(s)\).
We begin representing \(n\)-bundles in this way in Figure \ref{fig:241024_path}, and we always illustrate them using a dash-dot line style.
If the \(n\)-bundle is disjoint from all punctures (which is most often the case for us), we naturally draw its representative arc with endpoints on \(\ell(s)\backslash L_{m,n}\).
If the \(n\)-bundle contains exactly one puncture, we draw the representative arc with one endpoint on that puncture and the other endpoint on \(\ell(s)\backslash L_{m,n}\).
(This comes up only once for us: in Figure \ref{fig:241024_path}, and we provide a zoomed-in view in that case for clarity.)

The rest of this section will demonstrate that \({({(\sigma^+,\sigma^-)},\ell)}\) satisfies the 2-connected condition. 
This will require finding a collection of \(n\)-bundles contained in \(\sigma^+\) by following the isotopy of \(\sigma^+(2m-1)\subset H(2m-1)\) down to level \(0\).
Since \(m\) is unbounded above, the height of \(L_{m,n}\) can be arbitrarily large, but Lemma \ref{lem:inductive} allows us to jump from the top all the way down to level 2.
The lemma does not give us a picture of all of \(\sigma^+(2)\), only three particular \(n\)-bundles contained in \(\sigma^+(2)\), but that turns out to be all we will need.

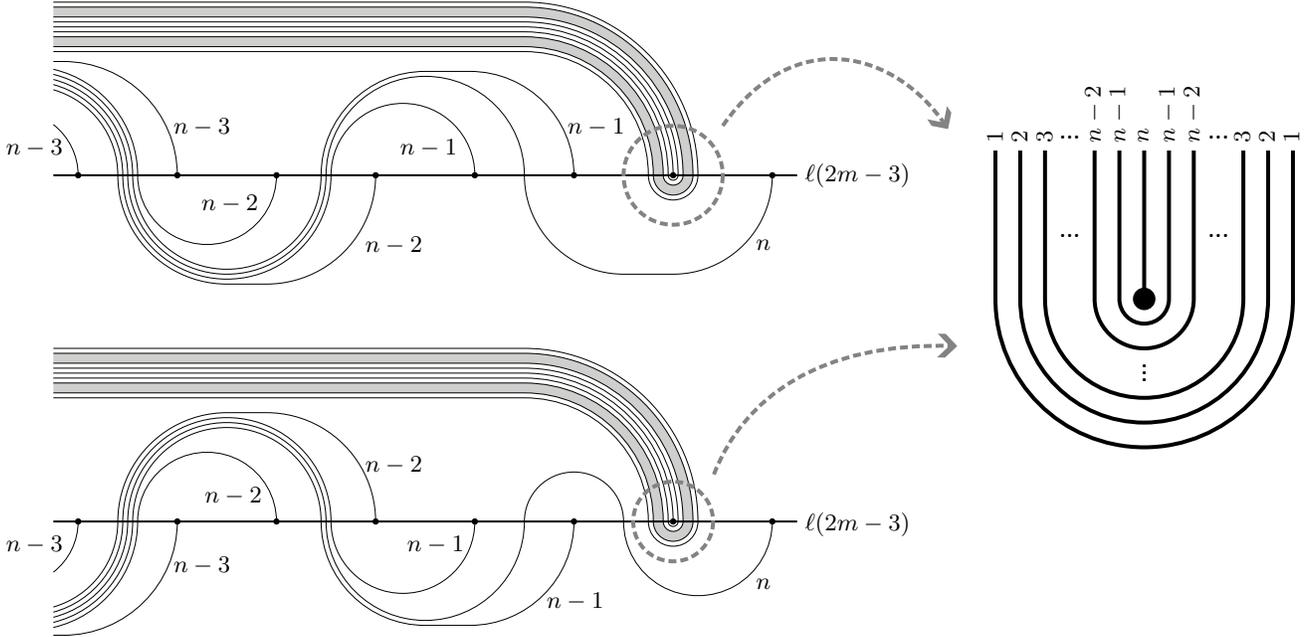
\begin{figure}
\labellist
\small\hair 2pt
\pinlabel \rotatebox{90}{\(1\)} [b] at 1350 700
\pinlabel \rotatebox{90}{\(2\)} [b] at 1385.4 700
\pinlabel \rotatebox{90}{\(3\)} [b] at 1420.8 700
\pinlabel \rotatebox{90}{\(n-2\)} [b] at 1491.7 700
\pinlabel \rotatebox{90}{\(n-1\)} [b] at 1527.1 700
\pinlabel \rotatebox{90}{\(n\)} [b] at 1562.5 700
\pinlabel \rotatebox{90}{\(n-1\)} [b] at 1597.9 700
\pinlabel \rotatebox{90}{\(n-2\)} [b] at 1633.3 700
\pinlabel \rotatebox{90}{\(3\)} [b] at 1704.2 700
\pinlabel \rotatebox{90}{\(2\)} [b] at 1739.6 700
\pinlabel \rotatebox{90}{\(1\)} [b] at 1775 700

\pinlabel {\(\ell(2m-3)\)} [l] at 1070 660
\pinlabel {\(\ell(2m-3)\)} [l] at 1070 160

\pinlabel {\(n-3\)} [l] at -75 700
\pinlabel {\(n-3\)} [l] at 165 730
\pinlabel {\(n-2\)} [l] at 205 620
\pinlabel {\(n-2\)} [l] at 440 560
\pinlabel {\(n-1\)} [l] at 490 700
\pinlabel {\(n-1\)} [l] at 730 730
\pinlabel {\(n\)} [l] at 1000 560

\pinlabel {\(n-3\)} [l] at -75 130
\pinlabel {\(n-3\)} [l] at 165 100
\pinlabel {\(n-2\)} [l] at 210 200
\pinlabel {\(n-2\)} [l] at 440 245
\pinlabel {\(n-1\)} [l] at 700 50
\pinlabel {\(n-1\)} [l] at 500 130
\pinlabel {\(n\)} [l] at 1000 75

\endlabellist
\centering
\includegraphics[width=1\textwidth]{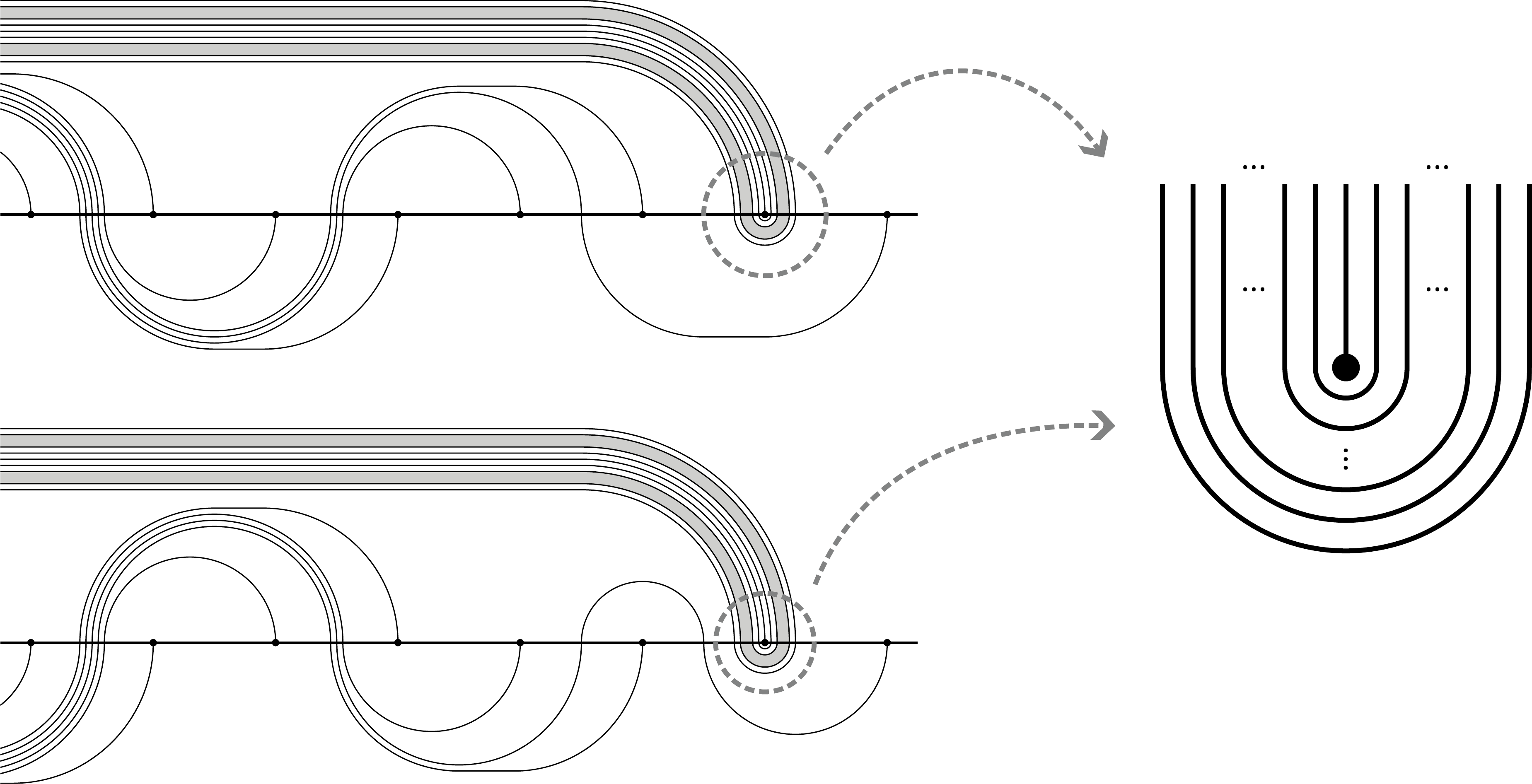}
\caption{This is essentially a recreation of Figure 12 from \cite{jang2016knot}.
Above on the left, we see the right part of \(\sigma^+(2m-3)\) when \(n\) is even, and below that, we see the right part of \(\sigma^+(2m-3)\) when \(n\) is odd.
Any arc labeled with an integer \(i\) is a subarc of \(\sigma^+_i(2m-3)\).}
\label{fig:241022_level_2m-3}
\end{figure}

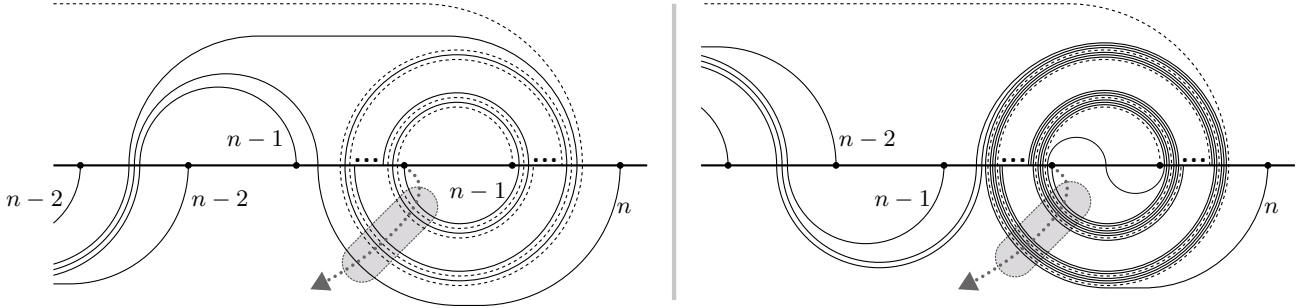
\begin{figure}
\labellist
\small\hair 2pt
\pinlabel {\(n-2\)} [r] at 20 140
\pinlabel {\(n-2\)} [l] at 175 140
\pinlabel {\(n-1\)} [r] at 310 220
\pinlabel {\(n-1\)} [br] at 600 135
\pinlabel {\(n\)} at 750 130

\pinlabel {\(n-2\)} [l] at 1025 220
\pinlabel {\(n-1\)} [r] at 1160 140
\pinlabel {\(n\)} at 1600 130

\endlabellist
\centering
\includegraphics[width=1\textwidth]{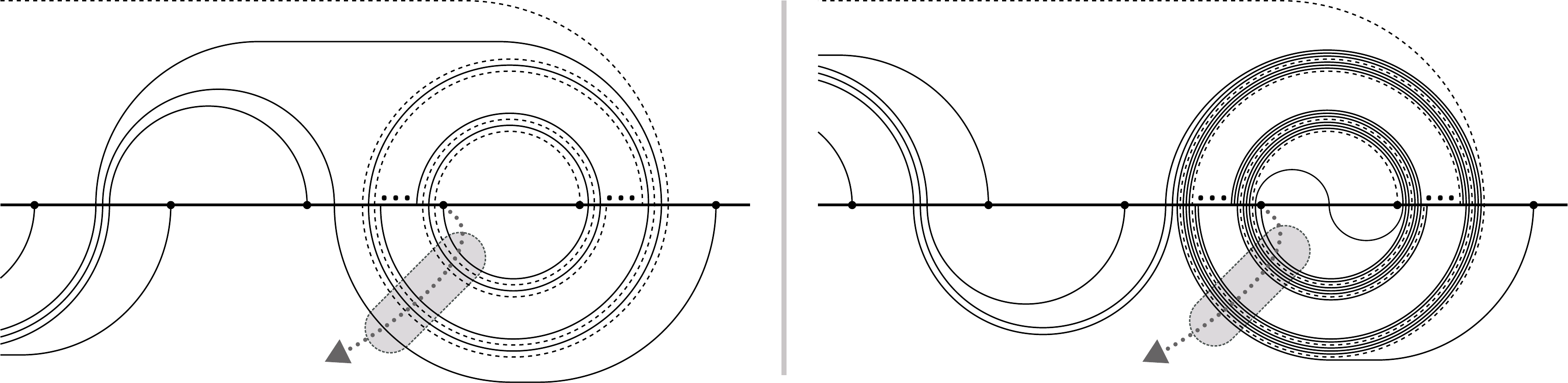}
\caption{We see \(\sigma^+(2m-3.5)\) when \(n\) is even (left) and odd (right).
The horizontal line represents \(\ell(2m-3.5)\).
The dashed arc illustrated as having a puncture as one of its endpoints represents a whole collection of arcs of \(\sigma^+(2m-3.5)\).
At the puncture, the picture looks like the right side of Figure \ref{fig:241022_level_2m-3}.
Away from the puncture, the dashed arc represents \(2n-1\) parallel strands.
A detailed image of the gray oval is found in Figure \ref{fig:241024_sequences}.}
\label{fig:241031_level-2m-minus-3point5}
\end{figure}

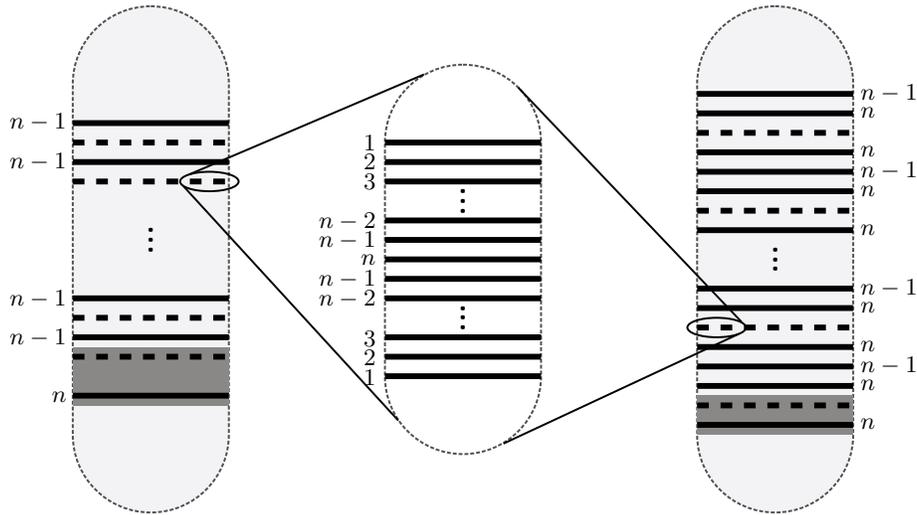
\begin{figure}
\labellist
\small\hair 2pt
\pinlabel {\(n\)} [r] at 0 85
\pinlabel {\(n-1\)} [r] at 0 127.9
\pinlabel {\(n-1\)} [r] at 0 156.4
\pinlabel {\(n-1\)} [r] at 0 256.4
\pinlabel {\(n-1\)} [r] at 0 285

\pinlabel {1} [r] at 225 99
\pinlabel {2} [r] at 225 113.3
\pinlabel {3} [r] at 225 127.5
\pinlabel {\(n-2\)} [r] at 225 156
\pinlabel {\(n-1\)} [r] at 225 170.3
\pinlabel {\(n\)} [r] at 225 184.5
\pinlabel {\(n-1\)} [r] at 225 198.8
\pinlabel {\(n-2\)} [r] at 225 213
\pinlabel {3} [r] at 225 241.5
\pinlabel {2} [r] at 225 255.8
\pinlabel {1} [r] at 225 270

\pinlabel {\(n\)} [l] at 570 65
\pinlabel {\(n\)} [l] at 570 93.4
\pinlabel {\(n-1\)} [l] at 570 107.5
\pinlabel {\(n\)} [l] at 570 121.7
\pinlabel {\(n\)} [l] at 570 150.1
\pinlabel {\(n-1\)} [l] at 570 164.2
\pinlabel {\(n\)} [l] at 570 206.8
\pinlabel {\(n\)} [l] at 570 235.1
\pinlabel {\(n-1\)} [l] at 570 249.3
\pinlabel {\(n\)} [l] at 570 263.5
\pinlabel {\(n\)} [l] at 570 291.8
\pinlabel {\(n-1\)} [l] at 570 306

\endlabellist
\centering
\includegraphics[width=.63\textwidth]{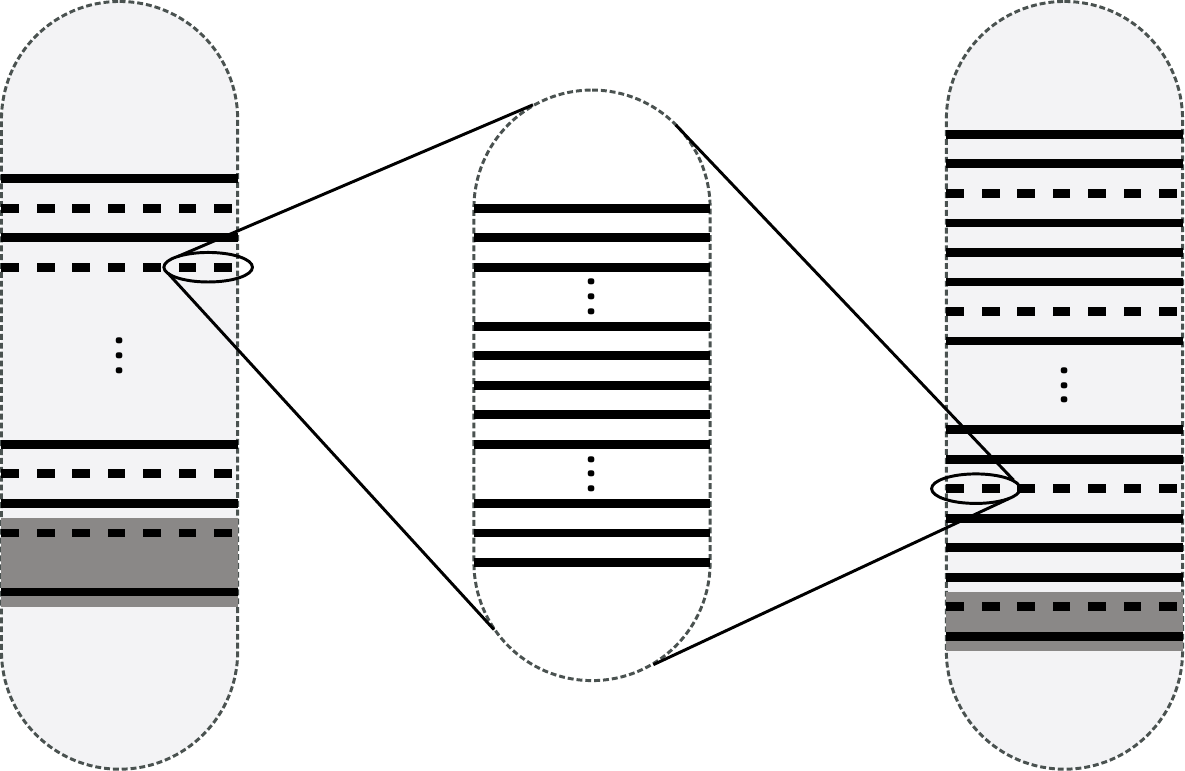}
\caption{On the left (for even \(n\)) and right (for odd \(n\)) are more detailed images of of the gray oval regions in Figure \ref{fig:241031_level-2m-minus-3point5}.
Each dashed line represents the collection of parallel arcs of \(\sigma^+(2m-3.5)\) illustrated in the middle picture.
Most importantly, the dark gray band at the bottom highlights an \(n\)-bundle.}
\label{fig:241024_sequences}
\end{figure}

\begin{figure}
\labellist
\small\hair 2pt
\pinlabel {\(\ell(2m-4)\)} [r] at 0 143
\pinlabel {\(\ell(2m-3.5)\)} [r] at 0 428
\endlabellist
\centering
\includegraphics[width=.5\textwidth]{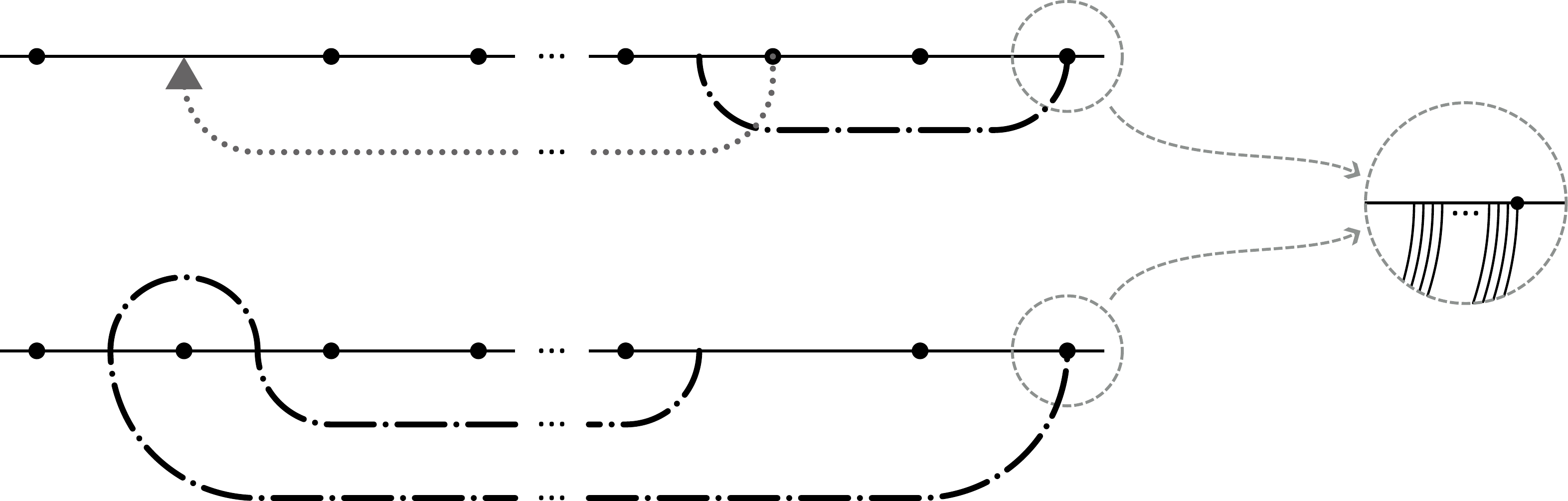}
\caption{In this and the following figures, curves drawn with a dash-dot pattern are \(n\)-bundles.
The gray dotted line with an arrow represents the movement of one of the punctures under the isotopy of the sphere from level \(2m-3.5\) down to level \(2m-4\).
For each level \(s\) pictured here, the leftmost puncture shown is \(p_1(s)\) and the right most puncture shown is \(p_{2n}(s)\).
Importantly, the \(n\)-bundle found in \(H_-(2m-3.5)\) gives rise to an \(n\)-bundle in \(H_+(2m-4)\).}
\label{fig:241024_path}
\end{figure}

\begin{figure}
\labellist
\small\hair 2pt
\pinlabel {\(\ell(s_0)\)} [l] at 790 720
\pinlabel {\(\ell(s_0-1.5)\)} [l] at 790 288.7
\pinlabel {\(\ell(s_0-1)\)} [l] at 790 504.3
\pinlabel {\(\ell(s_0-2)\)} [l] at 790 73
\endlabellist
\centering
\includegraphics[width=.4\textwidth]{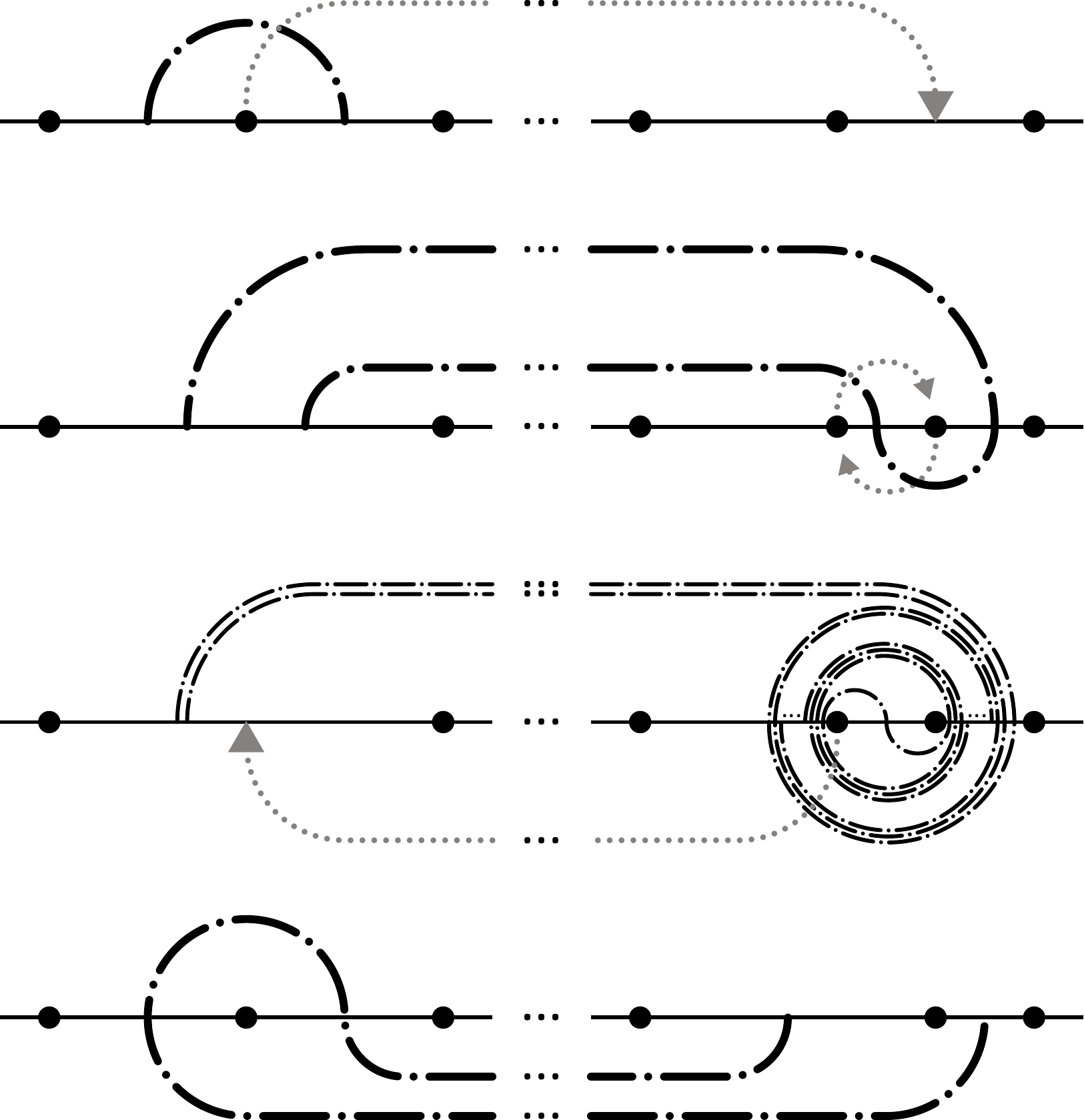}
\caption{Illustration of the inductive step for Lemma \ref{lem:inductive}.
In each of the four levels \(s\) pictured here, the leftmost puncture is \(p_1(s)\), but the three or four punctures to the right of the ellipsis are not the rightmost punctures along \(\ell(s)\).
In the third picture, the \(n\)-bundles are drawn with a thinner stroke width in order to properly show the spiraling around the two punctures.
In the bottom picture we see three resulting \(n\)-bundles, one in \(H_+(s_0-2)\) and two in \(H_-(s_0-2)\).}
\label{fig:241024_path2}
\end{figure}

\begin{lemma}\label{lem:inductive}
The bridge sphere \(H(2)\) contains the following three \(n\)-bundles:
\begin{enumerate}
    \item an \(n\)-bundle in \(H^+(2)\) with endpoints on the interiors of \(\ell_1(2)\) and \(\ell_2(2)\),
    \item an \(n\)-bundle in \(H^-(2)\) with endpoints on the interiors of \(\ell_2(2)\) and \(\ell_4(2)\). 
    \item an \(n\)-bundle in \(H^-(2)\) with endpoints on the interiors of \(\ell_1(2)\) and \(\ell_5(2)\). 
\end{enumerate}
\end{lemma}

We will use a reverse induction proof to find the first required \(n\)-bundle.
Then we will see that the existence of the second and third \(n\)-bundles quickly follows.

\begin{proof}
For a height \(s\in\{2, 4, 6, \hdots, 2m-4\}\), let \(P(s)\) be the following statement.
\textit{The hemisphere \(H^+(s)\) contains an \(n\)-bundle with one endpoint on \(\ell_1(s)\) and its other endpoint on \(\ell_2(s)\).}
We first show that \(P(2m-4)\) is true.
Then we show that for an even height \(s_0\), if \(P(s_0)\) is true, then \(P(s_0-2)\) follows, which allows us to conclude that \(P(2)\) is true by induction.

We start by establishing the base case, which is that \(P(2m-4)\) is true.
Since the top part of our link \(L_{m,n}\) matches the top part of \(K_n\) from \cite{jang2016knot}, we will start with Figure 12 of \cite{jang2016knot}, which illustrates \(\sigma^+(2m-3)\) (which for \cite{jang2016knot} is specifically \(\sigma^+(3)\)).
For our purposes, we only need the right side of their picture, which we reproduce here in Figure \ref{fig:241022_level_2m-3}.
Note that the picture is different depending on the parity of \(n\).

To move from level \(2m-3\) down to level \(2m-3.5\), the second-to-last and third-to-last punctures twist around each other.
This is illustrated in Figure \ref{fig:241031_level-2m-minus-3point5}, and a more detailed view of an important portion is shown in Figure \ref{fig:241024_sequences}, in which we find an \(n\)-bundle in \(\sigma^+(2m-3.5)\).

To move from level \(2m-3.5\) to level \(2m-4\), the third-to-last puncture moves leftward in front of all the other punctures save the first, thereby becoming the second puncture from the left at level \(2m-4\), as shown in Figure \ref{fig:241024_path}.
Note that we only focus on and illustrate what happens to one \(n\)-bundle and ignore the rest of \(\sigma^+(2m-3.5)\).
We have shown the base case \(P(2m-4)\) is true.

For the inductive step, consider Figure \ref{fig:241024_path2}.
Fix \(s_0\in\{4, 6, \hdots, 2m-4\}\) and assume \(P(s_0)\) is true, so \(\sigma^+(s_0)\cap H_+(s_0)\) contains the required \(n\)-bundle.
As the bridge sphere is isotoped down to level \(s_0-1\), the second puncture along \(\ell(s_0)\) follows its corresponding strand of \(L_{m,n}\), moving behind some number of other strands until it takes its place between two punctures some distance to the right.
Observe that this action pulls the \(n\)-bundle with it.
Then when moving from level \(s_0-1\) to level \(s_0-1.5\), that puncture and the one to its left perform some positive number of full twists around each other, and around them the \(n\)-bundle is forced to twist into a tao shape.
This results in some positive number of \(n\)-bundles in \(H_-(s_0-1.5)\) which straddle the two punctures involved in the twist region.
We focus on only one of these \(n\)-bundles and ignore the rest of \(\sigma^+(s_0-1.5)\).
Next, we move from level \(s_0-1.5\) to \(s_0-2\).
During this itotopy, of the two punctures involved in the twist region, the one which started and ended on the left moves leftward in front of all the other punctures to become the second puncture from the left at level \(s_0-2\), pulling the \(n\) bundle with it and resulting in the desired \(n\)-bundle at level \(s_0-2\), finishing the induction and allowing us to conclude that \(P(2)\) is true.
Thus we have the first \(n\)-bundle required in the statement of the lemma.

To find the second and third required \(n\)-bundles, observe that in the induction step, when \(s_0=4\), the ellipsis shown at the bottom of Figure \ref{fig:241024_path2} represents a segment of \(\ell(2)\) containing no punctures, and the six punctures in the picture are \(p_1(2)\), \(p_2(2)\), \(p_3(2)\), \(p_4(2)\), \(p_5(2)\), and \(p_6(2)\).
Then the two \(n\)-bundles shown in \(H_-(2)\) fulfill the requirements for the second and third \(n\)-bundles in the statement of this lemma. 
\end{proof}

Now we have three \(n\)-bundles in \(H(2)\), which will presently allow us to get a picture of a number of \(n\)-bundles in \(H\).
When \(n\) is even, we only need the first and second \(n\)-bundles from Lemma \ref{lem:inductive}, and when \(n\) is odd, we only need the first and third.
(These are shown in Figure \ref{fig:241031_2-to-1-path}.)
We will see that both the process and the result of the isotopy from level 2 to level 0 is the same as in \cite{jang2016knot}.
For completion, we will reexplain here as Proposition \ref{prop:vertically-unperturbed}.

\begin{figure}
\labellist
\small\hair 2pt
\pinlabel {\(\ell(2)\)} [l] at 1915 297
\pinlabel {\(\ell(2)\)} [l] at 1915 85
\endlabellist
\centering
\includegraphics[width=.8\textwidth]{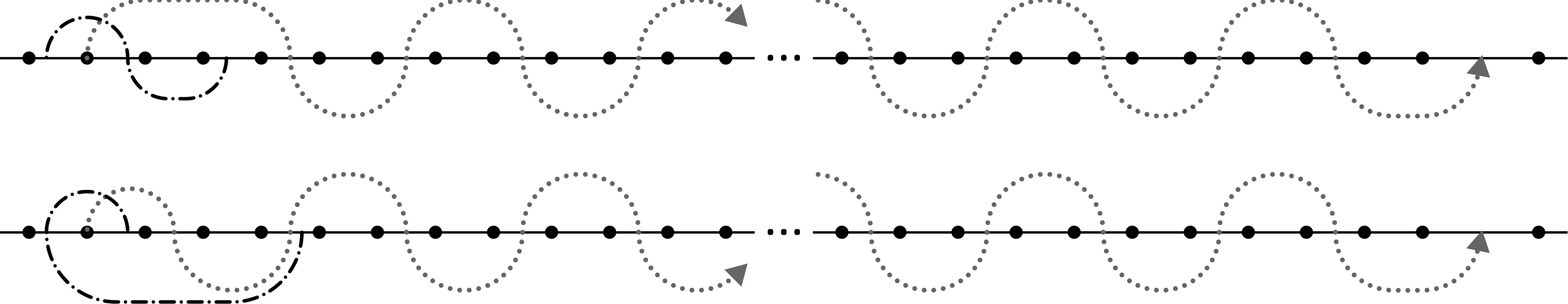}
\caption{Above, \(n\) even. 
Below, \(n\) odd.
Here we see two \(n\)-bundles at level 2, guaranteed by Lemma \ref{lem:inductive}.
The gray dotted line shows the path that the second puncture will take when the bridge sphere is isotoped down to level 1.
Here and in the following figures, the leftmost and rightmost punctures shown are \(p_1(s)\) and \(p_{2n}(s)\), respectively.
}
\label{fig:241031_2-to-1-path}
\end{figure}

\begin{figure}
\labellist
\small\hair 2pt
\pinlabel {\(\ell(1)\)} [l] at 1915 310
\pinlabel {\(\ell(1)\)} [l] at 1915 100
\endlabellist
\centering
\includegraphics[width=.8\textwidth]{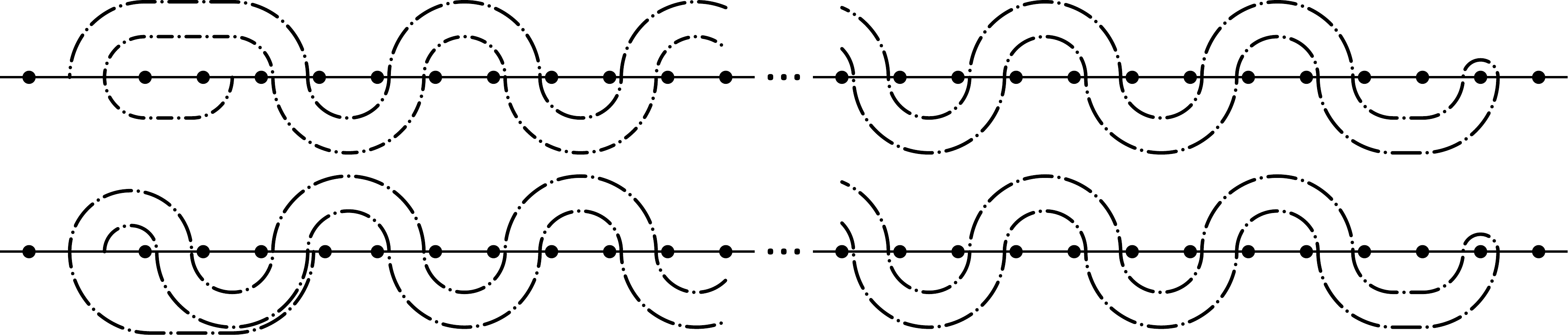}
\caption{Above, \(n\) even. 
Below, \(n\) odd.
The isotopy of the bridge sphere from level 2 to level 1 results in this collection of \(n\)-bundles.}
\label{fig:241031_level-1}
\end{figure}

\begin{figure}
\labellist
\small\hair 2pt
\pinlabel {\(\ell(1)\)} [l] at 1915 290
\pinlabel {\(\ell(1)\)} [l] at 1915 78
\endlabellist
\centering
\includegraphics[width=.8\textwidth]{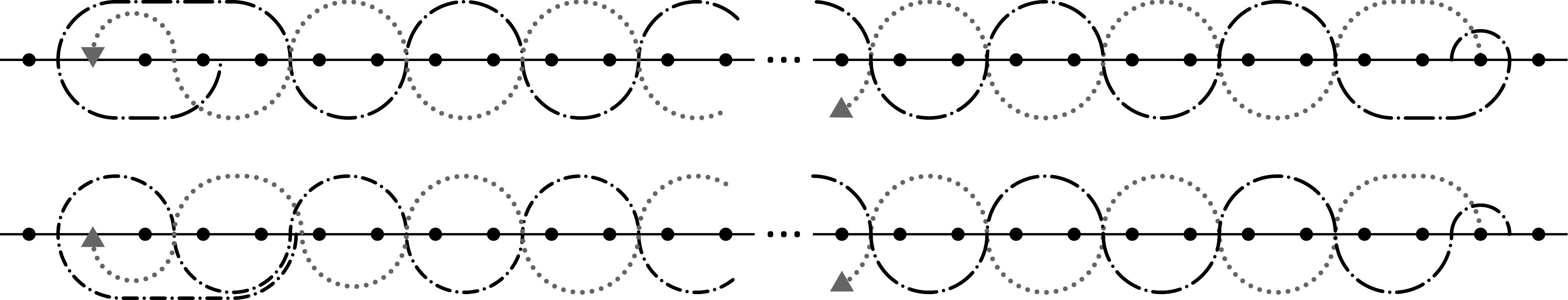}
\caption{Above, \(n\) even. 
Below, \(n\) odd.
In this figure we ignore roughly half of the \(n\)-bundles shown in Figure \ref{fig:241031_level-1}.
The gray dotted line shows the path that the second-to-last puncture shown will take when the bridge sphere is isotoped down to level 0.}
\label{fig:241031_level-1-and-path}
\end{figure}

\begin{figure}
\labellist
\small\hair 2pt
\pinlabel {\(\ell\)} [l] at 1945 295
\pinlabel {\(\ell\)} [l] at 1945 78
\endlabellist
\centering
\includegraphics[width=.8\textwidth]{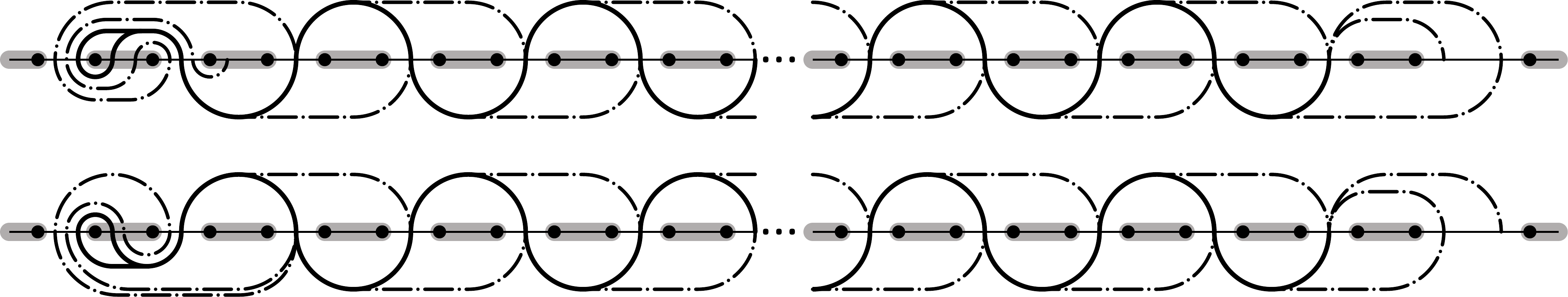}
\caption{Above, \(n\) even. 
Below, \(n\) odd.
The solid black curves represent collections of multiple parallel \(n\)-bundles.
The arcs of \(\ell\) highlighted in gray are \(\sigma_1^-, \sigma_2^-, \hdots, \sigma_n^-\).}
\label{fig:241031_level-0}
\end{figure}

\begin{figure}
\labellist
\small\hair 2pt
\pinlabel {\(\ell\)} [l] at 1945 145
\pinlabel {\(\ell\)} [l] at 1945 73
\endlabellist
\centering
\includegraphics[width=.8\textwidth]{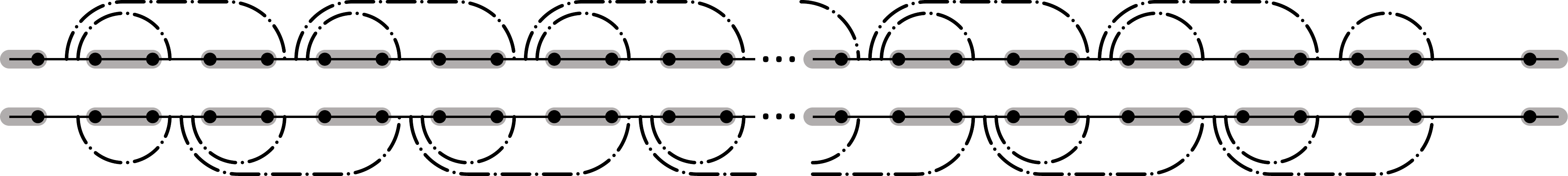}
\caption{Above, selected \(n\)-bundles in \(H_+\) for \(n\) even.
Below, selected \(n\)-bundles in \(H_-\) for \(n\) even.
Here and in Figure \ref{fig:241031_sigma-plus-odd}, the arcs of \(\ell\) highlighted in gray are \(\sigma_1^-, \sigma_2^-, \hdots, \sigma_n^-\).}
\label{fig:241031_sigma-plus-even}
\end{figure}

\begin{figure}
\labellist
\small\hair 2pt
\pinlabel {\(\ell\)} [l] at 1945 145
\pinlabel {\(\ell\)} [l] at 1945 73
\endlabellist
\centering
\includegraphics[width=.8\textwidth]{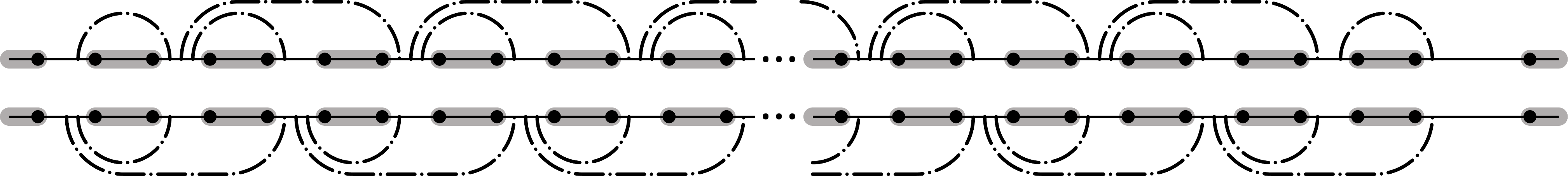}
\caption{Above, selected \(n\)-bundles in \(H_+\) for \(n\) odd.
Below, selected \(n\)-bundles in \(H_-\) for \(n\) odd.}
\label{fig:241031_sigma-plus-odd}
\end{figure}

\begin{prop}\label{prop:vertically-unperturbed}
The link \(L_{m,n}\) is unperturbed with respect to the \(n\)-bridge sphere \(H\).
\end{prop}

\begin{proof}
We want to show that \(\left(\left(\sigma^+,\sigma^-\right),\ell\right)\) satisfies the 2-connected condition, so we need to take the \(n\)-bundles we have found at level 2 and observe the isotopy down to level 0.
Moving from level 2 to level 1, \(p_2(2)\) will weave rightward between the other punctures and end up as \(p_{2n-1}(1)\).
This path is shown in figure \ref{fig:241031_2-to-1-path}.
This will pull the \(n\)-bundle in \(H_+(2)\) along the path, resulting in the collection of \(n\)-bundles illustrated in Figure \ref{fig:241031_level-1}.
To simplify, we will ignore roughly half of these and focus only on the \(n\)-bundles pictured in Figure \ref{fig:241031_level-1-and-path}, which also illustrates the path that the penultimate puncture will take when \(H(1)\) is isotoped to \(H\).
The result of this isotopy is a large amount of \(n\)-bundles in \(H\), all pictured in Figure \ref{fig:241031_level-0}.
Of those, we pick out the specific collection of \(n\)-bundles illustrated in Figure \ref{fig:241031_sigma-plus-even} for the case when \(n\) is even and in Figure \ref{fig:241031_sigma-plus-odd} for when \(n\) is odd.

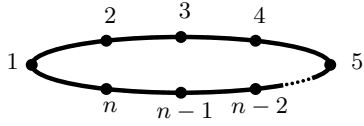
\begin{figure}
\labellist
\small\hair 2pt
\pinlabel {1} at -50 145
\pinlabel {2} at 321.5 320
\pinlabel {3} at 606 340
\pinlabel {4} at 890.5 320
\pinlabel {5} at 1262 145
\pinlabel {\(n\)} at 321.5 -30
\pinlabel {\(n-1\)} at 606 -50
\pinlabel {\(n-2\)} at 890.5 -30
\endlabellist
\centering
\includegraphics[width=.25\textwidth]{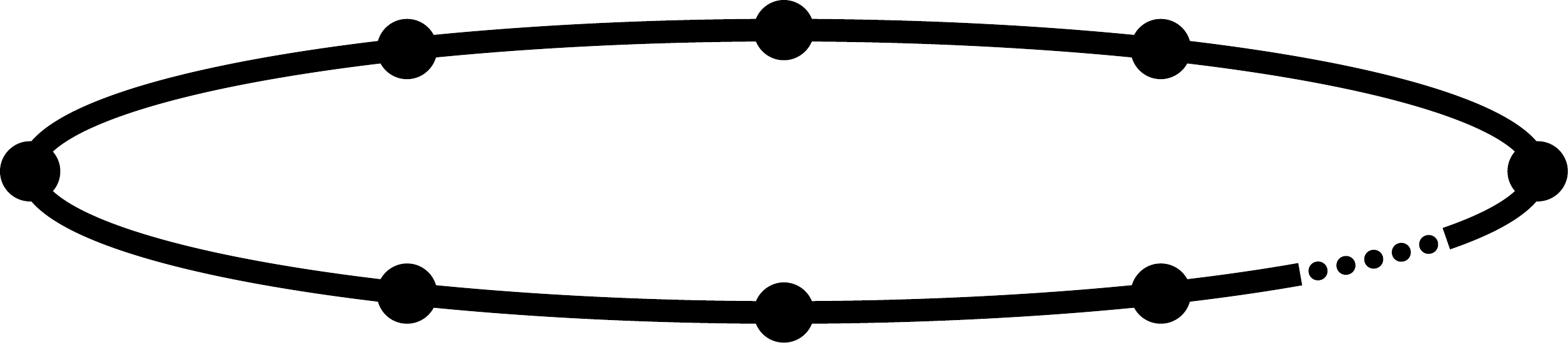}
\caption{For all \(\varepsilon\in\{+,-\}\) and all \(i,j\in\{1,2,\hdots,n\}\) with \(i\neq j\), the graph \(\mathcal{G}_{i,j,\varepsilon}\) contains this subgraph.}
\label{fig:241113_graph-cycle}
\end{figure}

Fix \(\varepsilon\in\{+,-\}\), and fix \(i,j\in\{1,2,\hdots,2n\}\), with \(i\neq j\).
Figures \ref{fig:241031_sigma-plus-even} and \ref{fig:241031_sigma-plus-odd} show that \(\sigma_i^-\) and \(\sigma_j^-\) are separated in \(H_{\varepsilon}\) by at least one \(n\)-bundle, regardless of the parity of \(n\).
It follows from the definition of an \(n\)-bundle that for any \(v\in\{1,2,\hdots,n\}\), vertex \(v\) in the graph \(\mathcal{G}_{i,j,\varepsilon}\) is connected to vertex \(v+1\) (using addition mod \(n\)), so therefore \(\mathcal{G}_{i,j,\varepsilon}\) contains a circuit passing in order through the vertices 1, 2, 3, \(\hdots\), \(n\), and back to 1, as shown in Figure \ref{fig:241113_graph-cycle}.
Such a graph is necessarily 2-connected, regardless of what additional edges the graph may contain.
Since this is true for all \(i\), \(j\), and \(\varepsilon\), it follows that \(\left(\left(\sigma^+,\sigma^-\right),\ell\right)\) satisfies the 2-connected condition.

Therefore by Theorem \ref{thm:two-con-str-irr}, \(H\) is a strongly irreducible bridge sphere for \(L_{m,n}\).
It is known that strongly irreducible bridge spheres are unperturbed, so this completes the proof.
\end{proof}

\section{\(L_{m,n}\) is Horizontally Unperturbed}\label{sec:WRT-V}

\begin{figure}
\labellist
\small\hair 2pt
\pinlabel {Surgery} [b] at 380 235
\pinlabel {Isotopy} [b] at 907.5 235
\pinlabel {Isotopy} [b] at 1435 235
\pinlabel {\(x\)} [b] at 295 -10
\pinlabel {\(x\)} [b] at 820 -10
\pinlabel {\(x\)} [b] at 1345 -10
\pinlabel {\(x\)} [b] at 1870 -10
\pinlabel {\(x_0\)} [b] at 153 -40
\pinlabel {\(x_0\)} [b] at 678 -40
\pinlabel {\(x_0\)} [b] at 1203 -40
\pinlabel {\(x_0\)} [b] at 1728 -40
\endlabellist
\centering
\includegraphics[width=.85\textwidth]{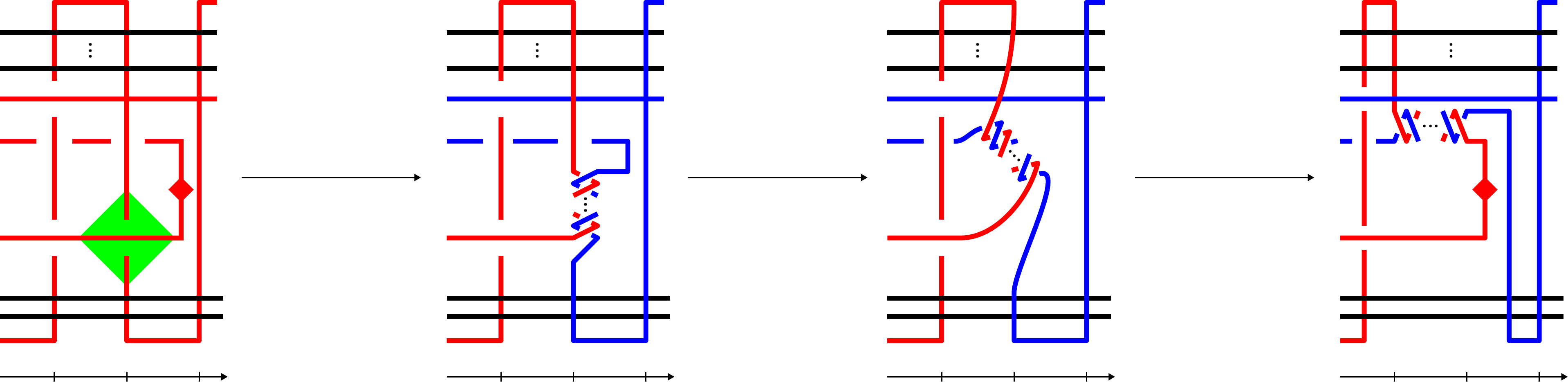}
\caption{Before and after the surgery and then the isotopy, each such region of the link contains exactly one \(V\)-maximum.
In these pictures we do not illustrate crossing information at most of the crossings because it varies depending on which distinguished crossing we are focusing on.
The number of black horizontal strands along the top of each picture may be any nonnegative even number, depending on the values of \(n\) and \(m\) and on which distinguished crossing we are focusing on.}
\label{fig:241010_rotating-twists}
\end{figure}

\begin{figure}
\labellist
\small\hair 2pt
\pinlabel {\(V\)} at 1300 1420
\endlabellist
\centering
\includegraphics[width=.25\textwidth]{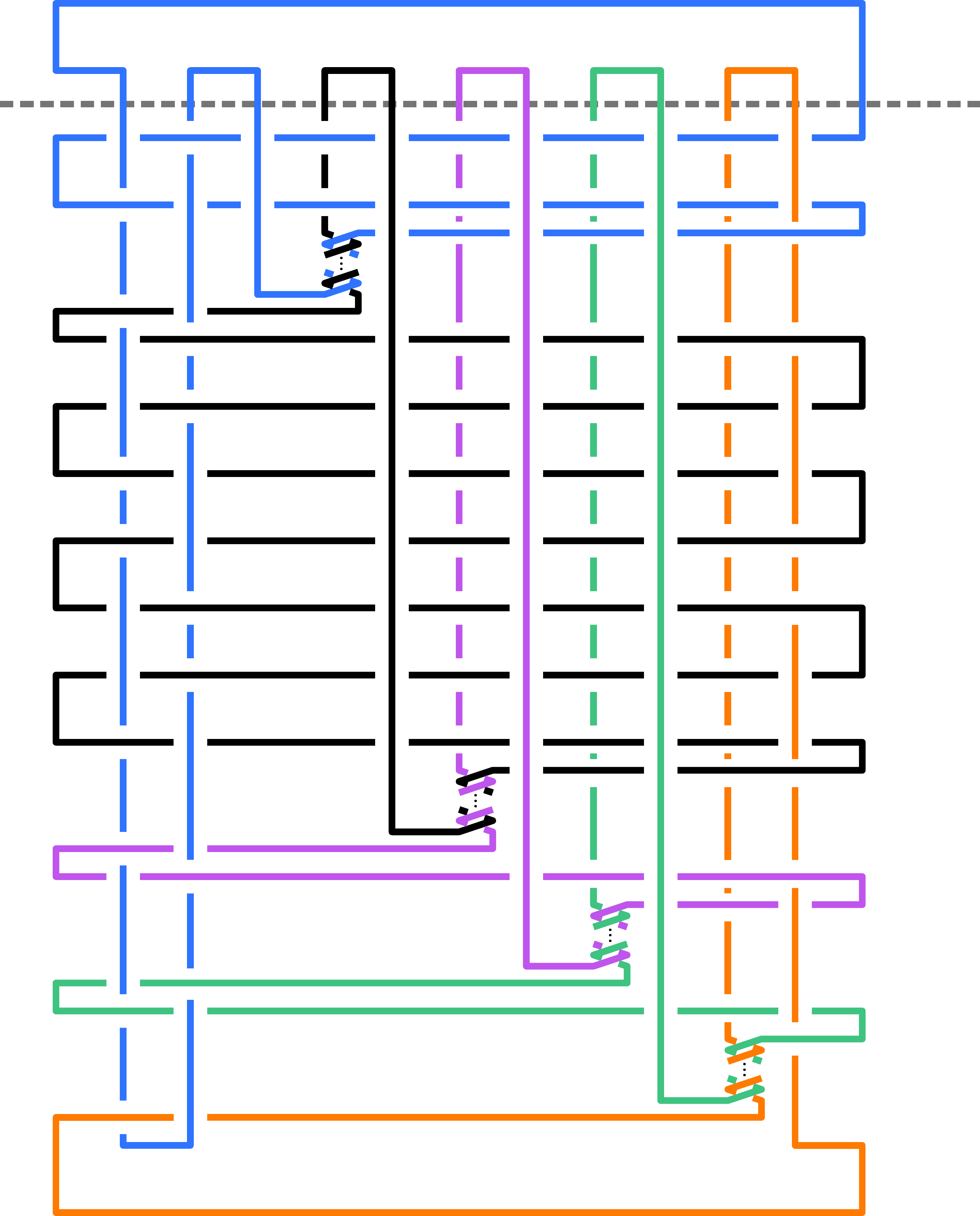}
\caption{This is the link \(L_{6,9}\) from Figure \ref{fig:241028_9-6-example} after the isotopy illustrated in Figure \ref{fig:241010_rotating-twists} is performed at each twist region, showing how \(V\) is a \(6\)-bridge sphere for \(L_{6,9}\).}
\label{fig:241016_9-6-example-rotated.pdf}
\end{figure}

\begin{prop}\label{prop:horizontally-unperturbed}
The link \(L_{m,n}\) is unperturbed with respect to the \(m\)-bridge sphere \(V\).
\end{prop}

\begin{proof}
We point out that after replacing crossings in \(K_{m,n}\) with twist regions, \(L_{m,n}\) is not in fact in an \(m\)-bridge position with respect to \(V\), for each twist region introduces multiple new \(V\)-extrema, all on the same side of \(V\).
An isotopy at each twist region easily solves this problem.

Let \(x_0\) be the \(x\)-value corresponding to one of the distinguished crossings of \(K_{m,n}\).
Consider the pictures in Figure \ref{fig:241010_rotating-twists}.
The first picture shows the part of \(K_{m,n}\) lying in the region \(X=(x_0-1.5,x_0+1.5)\times(-1,2m)\times(-1,1)\).
This part of the knot contains a single distinguished crossing.
The second picture in Figure \ref{fig:241010_rotating-twists} shows the result of the surgery replacing the crossing with a twist region.
The third and fourth pictures show an isotopy of the strands of \(L_{m,n}\) in \(X\).
This isotopy is local, happening only within \(X\).
The point is that in the left picture we see that \(X\) contains exactly one \(V\)-maximum, and in the right picture, i.e., after the surgery and isotopy, \(X\) still contains exactly one \(V\)-maximum.
We perform this surgery and isotopy in a similar region around each distinguished crossing, so in the end the number of extrema relative to \(h_x\) and their locations relative to \(V\) has not changed from \(K_{m,n}\) to \(L_{m,n}\).
Therefore \(V\) is an \(m\)-bridge sphere for \(L_{m,n}\).

Next, observe that \(L_{m,n}\) is an \(m-1\) component link. 
Of those components, \(m-2\) of them have exactly one \(V\)-maximum and one \(V\)-minimum, which means that individually they are unknots.
The sphere \(V\) is necessarily an unperturbed bridge sphere with respect to each of these components.
We call the other link component of \(L_{m,n}'\).
(This is the component of \(L_{m,n}\) which contains the two leftmost \(H\)-maxima.)
Figures \ref{fig:241010_5-2-knot-odd-case} and \ref{fig:241015_4-1-knot-even-case} explicitly show that \(L_{m,n}'=4_1\) when \(n\) is even, and \(L_{m,n}'=5_2\) when \(n\) is odd.
Since \(4_1\) and \(5_2\) are 2-bridge knots, it follows that \(V\) is an unperturbed bridge sphere for \(L_{m,n}'\).
Since \(V\) is an unperturbed bridge sphere for each individual component of \(L_{m,n}\), \(V\) is therefore an unperturbed bridge sphere for \(L_{m,n}\) itself.

\begin{figure}
\labellist
\small\hair 2pt
\pinlabel {Isotopy} [b] at 202 235
\pinlabel {Isotopy} [b] at 464 235
\endlabellist
\centering
\includegraphics[width=.5\textwidth]{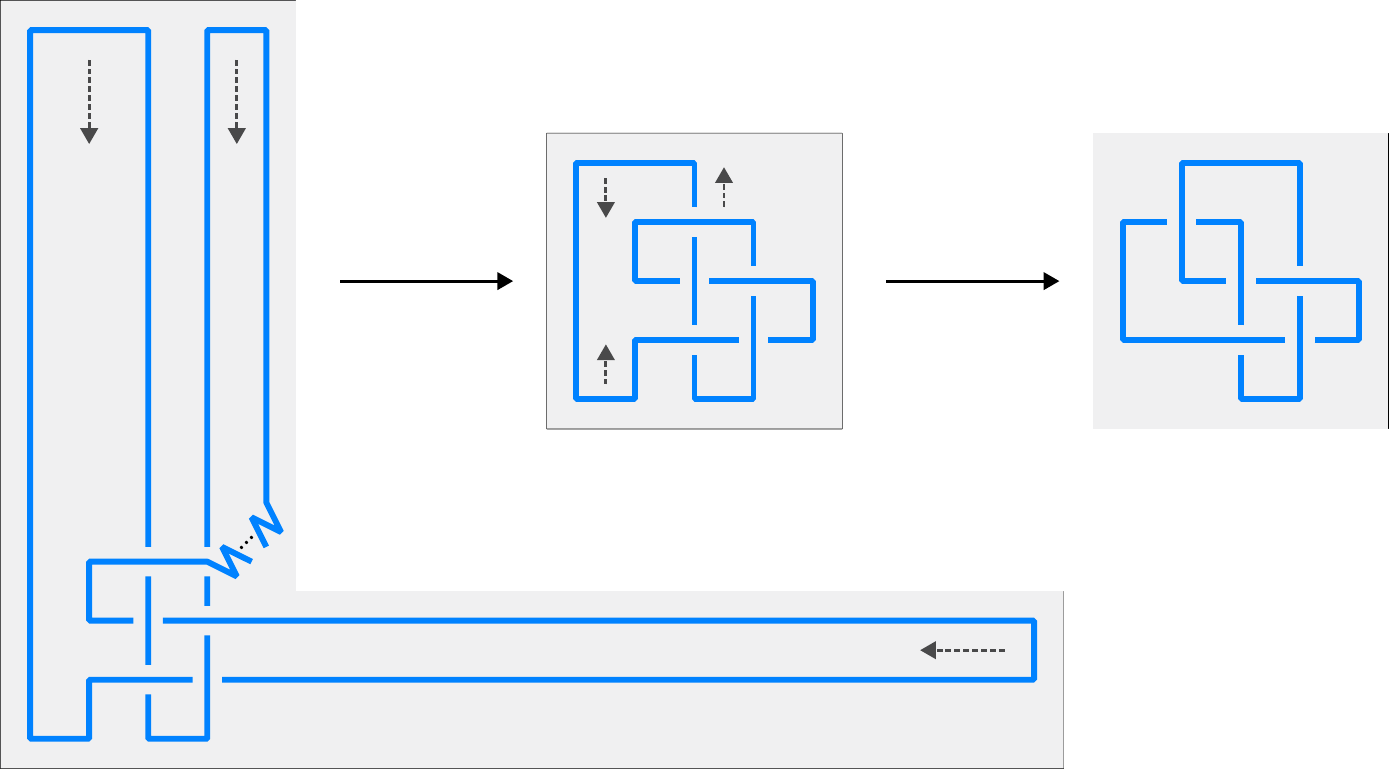}
\caption{When \(n\) is odd, \(L'_{m,n}\) is the \(5_2\) knot.}
\label{fig:241010_5-2-knot-odd-case}
\end{figure}

\begin{figure}
\labellist
\small\hair 2pt
\pinlabel {Isotopy} [b] at 202 235
\pinlabel {Isotopy} [b] at 464 235
\endlabellist
\centering
\includegraphics[width=.5\textwidth]{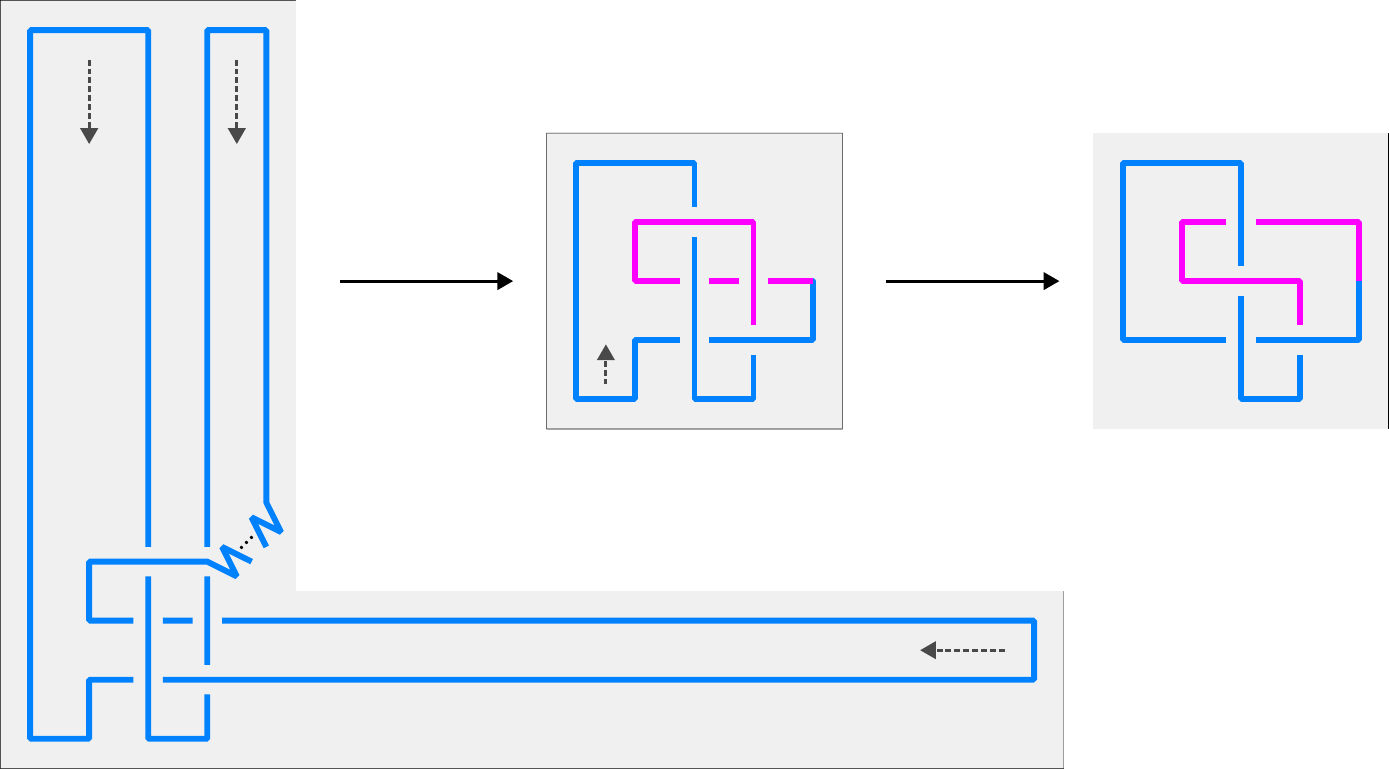}
\caption{When \(n\) is even, \(L'_{m,n}\) is the \(4_1\) knot.}
\label{fig:241015_4-1-knot-even-case}
\end{figure}

Now since \(V\) is unperturbed with respect to each component of \(L_{m,n}\) it follows that \(V\) is unperturbed with respect to \(L_{m,n}\) itself.
\end{proof}

Finally we can put all the pieces together to prove our main result.

\begin{proof}[Proof of Theorem \ref{thm:main}]
Fix integers \(m\) and \(n\) such that \(3<m<n\), and let \(L_{m,n}\), \(H\), and \(V\) be constructed as in section \ref{sec:links}.
By Propositions \ref{prop:vertically-unperturbed} and \ref{prop:horizontally-unperturbed}, both \(H\) and \(V\) are unperturbed bridge spheres for \(L_{m,n}\), and since \(n\neq m\), \(H\) and \(V\) have different numbers of punctures from each other, so they cannot be isotopic.
\end{proof}

\section{Final Remarks}
First, the reader may have noticed that our proofs that \(H\) and \(V\) are unperturbed do not rely on the assumption that \(n\neq m\).
The reason Theorem \ref{thm:main} assumes \(n\neq m\) is that it could be the case that \(H\) and \(V\) are isotopic if \(n=m\), in which case our examples would not be new; others have already found bridge positions with arbitrarily high bridge number.

Second, as described in section \ref{subsec:moves}, the only moves typically allowed in relating two bridge positions are perturbation and de-perturbation moves, which change a bridge position's bridge number by 1 and \(-1\), respectively.
As indicated by Theorem \ref{thm:main} and by the other results we mentioned in section \ref{subsec:moves}, increasing the bridge number is in some cases unavoidable when using these moves to go from one bridge position to another.
Interestingly, in the context of bridge positions or braid presentations, if we allow more types of elementary moves (in addition to the typical set), it may be possible to relate bridge positions or braid presentations without ever increasing the bridge number or braid index \cite{birman2006stabilization,menasco2024studying}. 
For example, computing the minimum number of steps required to relate bridge positions or braid presentations using Menasco et al.'s set of moves is an interesting problem worth investigating in the future.

Finally, a note about how we solved the problem of Proposition \ref{prop:horizontally-unperturbed}:
Arguments in the proof of Theorem 1.3 of \cite{blair2020wirtinger} told us that there must \textit{exist} an \(m\)-bridge sphere for for \(L_{m,n}\) since the link's diagram was colorable in a certain way, but initially we did not know that \(V\) itself was that sphere.
However, knowing that an \(m\)-bridge sphere \textit{exists} gave us the idea and intuition to see that simply rotating all the twist regions by a quarter turn as in Figure \ref{fig:241010_rotating-twists} would put \(L_{m,n}\) in the right position.
All that to say, even though the work in \cite{blair2020wirtinger} is not technically necessary to include in the written proof, it was necessary for our problem-solving process, and we would not have arrived at this proof without it.

\bibliography{ref}
\bibliographystyle{abbrv}

\end{document}